\documentclass[12pt]{amsart}

\usepackage{aliascnt}
\usepackage{amscd}
\usepackage{amsfonts}
\usepackage{amsmath}
\usepackage{amssymb}
\usepackage{amsthm}
\usepackage{appendix}
\usepackage{bm}
\usepackage{color}
\usepackage{chngpage}
\usepackage[shortlabels]{enumitem}
\usepackage{fancyhdr}
\usepackage{float}
\usepackage{fullpage}
\usepackage{graphics}
\usepackage{graphicx}
\usepackage[CJKbookmarks=true]{hyperref}
\usepackage[all]{hypcap}
\usepackage{indentfirst}
\usepackage{latexsym}
\usepackage{mathrsfs}
\usepackage{mathtools}
\usepackage{multirow}
\usepackage{pstricks}
\usepackage{rotating}
\usepackage{soul}
\usepackage{tikz}
\usepackage{tikz-network}
\usepackage{url}
\usepackage{xcolor}
\usepackage[all,poly,knot]{xy}

\theoremstyle{plain}
\newtheorem{theorem}{Theorem}[subsection]

\newtheorem*{acknowledgement*}{\protect \acknowledgementname}
\providecommand{\acknowledgementname}{Acknowledgement}
\newaliascnt{setup}{theorem}
\newtheorem{setup}[setup]{Setup}
\aliascntresetthe{setup}

\newaliascnt{question}{theorem}
\newtheorem{question}[question]{Question}
\aliascntresetthe{question}

\newaliascnt{lemma}{theorem}
\newtheorem{lemma}[lemma]{Lemma}
\aliascntresetthe{lemma}

\newaliascnt{assumption}{theorem}

\aliascntresetthe{assumption}

\newaliascnt{conjecture}{theorem}
\newtheorem{conjecture}[conjecture]{Conjecture}
\aliascntresetthe{conjecture}

\newaliascnt{proposition}{theorem}
\newtheorem{proposition}[proposition]{Proposition}
\aliascntresetthe{proposition}

\newaliascnt{corollary}{theorem}
\newtheorem{corollary}[corollary]{Corollary}
\aliascntresetthe{corollary}

\newaliascnt{problem}{theorem}

\aliascntresetthe{problem}

\theoremstyle{definition}
\newaliascnt{definition}{theorem}
\newtheorem{definition}[definition]{Definition}
\aliascntresetthe{definition}

\newaliascnt{example}{theorem}
\newtheorem{example}[example]{Example}
\aliascntresetthe{example}

\theoremstyle{remark}
\newaliascnt{remark}{theorem}
\newtheorem{remark}[remark]{Remark}
\aliascntresetthe{remark}

\newaliascnt{claim}{theorem}

\aliascntresetthe{claim}

\newaliascnt{fact}{theorem}

\aliascntresetthe{fact}

\newaliascnt{notation}{theorem}

\aliascntresetthe{notation}

\newaliascnt{remarks}{theorem}

\aliascntresetthe{remarks}

\DeclareMathOperator\Char{char}
\DeclareMathOperator\rmd{d}
\DeclareMathOperator\End{End}

\DeclareMathOperator\Fil{Fil}

\DeclareMathOperator\Gr{Gr}
\DeclareMathOperator\Hom{Hom}

\DeclareMathOperator\id{id}

\DeclareMathOperator\rank{rank}

\DeclareMathOperator\Spec{Spec}

\newcommand{\dR}{\mathrm{dR}}
\newcommand{\Hig}{\mathrm{Hig}}

\newcommand{\et}{\text{\'et}}
\newcommand{\MF}{\mathcal{MF}}
\newcommand{\MCF}{\mathrm{MCF}}
\newcommand{\MIC}{\mathrm{MIC}}
\newcommand{\MC}{\mathrm{MC}}

\numberwithin{equation}{section}
\definecolor{backgroundcolor}{RGB}{210,210,210}
\definecolor{wordscolor}{RGB}{50,50,50}

\def\Flow{\mathrm{Flow}}

\begin{document}

\title{Parabolic crystalline representations}

\author{Zhenmou Liu}
\email{\href{mailto:zhenmouliu@mail.ustc.edu.cn}{zhenmouliu@mail.ustc.edu.cn}}
\address{School of Mathematical Sciences, University of Science and Technology of China, Hefei, Anhui 230026, PR China}

\author{Jinbang Yang}
\email{\href{mailto:yjb@mail.ustc.edu.cn}{yjb@mail.ustc.edu.cn}}
\address{School of Mathematical Sciences, University of Science and Technology of China, Hefei, Anhui 230026, PR China}

\author{Kang Zuo}
\email{\href{mailto:zuok@uni-mainz.de}{zuok@uni-mainz.de}}
\address{School of Mathematics and Statistics, Wuhan University, Luojiashan, Wuchang, Wuhan, Hubei, 430072, P.R. China.; Institut f\"ur Mathematik, Universit\"at Mainz, Mainz 55099, Germany}

\begin{abstract}
The theory of crystalline representations was established by Fontaine and Laffaille, Faltings, and others. In this paper, we develop a parabolic version of this theory. The key point is the construction of the parabolic version of Fontaine-Faltings modules and Faltings’ $\mathbb D$-functor. The theory of Higgs-de Rham flows can be used to efficiently construct crystalline representations. We have established a parabolic version and utilized it to construct infinitely many crystalline representations. The twisted versions discussed in Sun, Yang, and Zuo's work can be seen as a special case, where the parabolic weights are equal at every infinity point.
\end{abstract}

\maketitle
\tableofcontents

\section{\bf Introduction}
The nonabelian Hodge theory, established by Hitchin and Simpson, associates semisimple flat bundles on a compact K\"ahler manifold with polystable Higgs bundles possessing vanishing Chern classes. As an analogue, Ogus and Vologodsky established the nonabelian Hodge theory in positive characteristic in \cite{OgVo07}. They constructed a pair of quasi-inverse functors, known as the Cartier functor and the inverse Cartier functor, between the category of nilpotent Higgs modules of exponent $\leq p-1$ and the category of nilpotent flat modules of exponent$\leq p-1$. Let $\mathcal{Y}$ be a proper smooth scheme over $W \coloneqq W(k)$ and $Y = \mathcal{Y} \times_W k$, where $k$ is a perfect field of characteristic $p>2$. Fontaine and Laffaille \cite{FoLa82} and Faltings \cite{Fal89} introduced the category $\mathcal{MF}^{\nabla}_{[0, p-2]}(\mathcal{Y}/W)$, whose objects are called the Fontaine-Faltings modules. These modules can be viewed as a $p$-adic analogue of polarized complex variations of Hodge structures. Faltings constructed a fully faithful functor $\mathbb{D}$ from the category of Fontaine-Faltings modules to the category of representations of the \'etale fundamental group of the generic fiber of $\mathcal{Y}$. This functor can be regarded as a $p$-adic analogue of the Riemann-Hilbert correspondence. The representations in the essential image of this functor are called crystalline representations.

Lan, Sheng, and Zuo \cite{LSZ19} introduced the notion of Higgs-de Rham flow. A Higgs-de Rham flow is a sequence of Higgs bundles and filtered de Rham bundles connected by the inverse Cartier functor and the grading functor. They established an equivalence of categories between the category of periodic Higgs-de Rham flows and the category of Fontaine-Faltings modules with endomorphism structures. By finding periodic Higgs bundles, one can obtain crystalline representations via the theory of Higgs-de Rham flows.

In \cite{SYZ22}, Sun, Yang, and Zuo introduced the concepts of twisted periodic Higgs–de Rham flows and twisted Fontaine-Faltings modules. They constructed an equivalence of categories between the category of twisted periodic Higgs-de Rham flows and the category twisted of Fontaine-Faltings module equipped with $W\left(\mathbb{F}_{p^f}\right)$-endomorphism structure, generalizing the work of Lan, Sheng, and Zuo. By investigating the selfmap induced by Higgs-de Rham flow on the moduli space of rank-2 stable logarithmic Higgs bundles of degree $1$ on $\mathbb{P}^1$ with marked points $D=\{x_1, x_2, \ldots, x_m\}$ for $m\geq 4$, they constructed infinitely many projective crystalline representations. This can be viewed as a special case of parabolic version.

In this paper, we establish the theory of parabolic Fontaine-Faltings modules and parabolic crystalline representations.  The key innovation lies in our global description of the pullback of parabolic de Rham bundles and the parabolic Frobenius pullback functor. Let $f: (Y', D_{Y'}) \to (Y, D_Y)$ be a morphism between smooth log schemes over $S$. Let $(V, \nabla)$ be a parabolic $\lambda$-flat bundle over $(Y, D_Y)/S$. The pullback of $(V, \nabla)$ is defined to be the $\lambda$-flat bundle
\[ f^*(V, \nabla) \coloneqq \bigcup_{\gamma \in \mathbb{Q}^n} f^*\big(V(-\gamma D)_0, \nabla(-\gamma D)_0\big) \otimes f^*\big(\mathcal{O}_Y(\gamma D), \lambda \cdot \rmd_Y(\gamma D)\big).\]
This enables us to construct the parabolic Frobenius pullback functor $\mathcal F_n$ over the truncated ring $W_n(k)$ directly in \autoref{sec_PFFandPHDF} as in \cite{LSZ19}.

By applying the parabolic Frobenius pullback functor $\mathcal F_n$, we introduce the notion of parabolic Fontaine-Faltings module over $W_n(k)$, and as a generalization of Faltings' $\mathbb{D}$-functor in \cite{Fal89} and $\mathbb D^{\log}$-functor in \cite{LYZ25}, we construct the $\mathbb D^{par}$-functor for parabolic Fontaine-Faltings modules, whose essential image consists of objects called parabolic crystalline representations.

\begin{theorem}[\autoref{thm_ParabolicDFunctor}]
Let $f$ be a positive integer and let $k$ is a perfect field of characteristic $p>2$ containing $\mathbb F_{p^f}$. Let $W$ be the ring of Witt vectors over $k$ and $W_n=W/p^nW$ with $S=\Spec W$, $S_n=\Spec W_n$. Let $Y$ be a smooth separated  scheme over $W$ with a geometrically connected generic fiber. Let $D_Y\subseteq Y$ be a relative normal crossing divisor $D_Y$ with complement $U_Y$. Denote by $(Y_n, D_{Y_n})\coloneqq (Y, D_Y)\times_SS_n$. Let $\mathcal {(Y,D_Y,U_Y)}$ be the $p$-adic formal completion of ${(Y,D_Y,U_Y)}$ with rigid generic fiber $(\mathcal Y_K,\mathcal{D}_{\mathcal Y_K},\mathcal{U}_{\mathcal Y_K} )$. Denote by $\mathcal Y_K^\circ\coloneqq\mathcal Y_K-\mathcal{D}_{\mathcal Y_K}$.
\begin{enumerate}
\item
There is a functor $\mathbb D^{par}$ from the category strict $p^n$-torsion parabolic Fontaine-Faltings modules over $(Y_n, D_{Y_n})/S_n$ (reps. torsion-free parabolic Fontaine-Faltings modules over $(\mathcal Y,\mathcal{D_Y})/S$)
to the category of continuous $\mathbb Z/p^n \mathbb Z$-representations (reps. $\mathbb Z_p$-representations) of $\pi_1^{\et}(\mathcal Y_K^\circ)$. This functor is compatible with usual Faltings' $\mathbb D$-functor in the following sense: for any strict $p^n$-torsion Fontaine-Faltings modules $M$ over $(Y_n,D_{Y_n})/S_n$ (reps. torsion-free parabolic Fontaine-Faltings modules $M$ over $(\mathcal Y,\mathcal{D_Y})/S$), there is a canonical isomorphism of continuous $\mathbb Z/p^n \mathbb Z$-representation (reps. $\mathbb Z_p$-representations) of $\pi_1^{\et}(\mathcal U_{\mathcal Y_K})$:
\[\mathbb D^{par}(M)\mid_{\mathcal U_{\mathcal Y_K}}\cong \mathbb D(M\mid_{ Y_n-D_{Y_n}}) \quad \left(\text{reps. }\mathbb D^{par}(M)\mid_{\mathcal U_{\mathcal Y_K}}\cong \mathbb D(M\mid_{ \mathcal Y-\mathcal{D_Y}})\right).\]

\item
If $Y$ is proper over $W$, then the functor in (1) can be algebraized  into a functor from the category strict $p^n$-torsion parabolic Fontaine-Faltings modules over $(Y_n, D_{Y_n})/S_n$ (reps. torsion-free parabolic Fontaine-Faltings modules over $(\mathcal Y,\mathcal{D_Y})/S$)
to the category of continuous $\mathbb Z/p^n \mathbb Z$-representations (reps. $\mathbb Z_p$-representations) of $\pi_1^{\et}(Y_K^\circ)$, where $Y_K^\circ$ is the generic fiber of $U_Y$.
\end{enumerate}
\end{theorem}

We also establish the theory of parabolic Higgs-de Rham flows over the truncated ring $W_n$. The modulo $p$ version of the theory of Higgs-de Rham flows has been extended to the context of parabolic bundles on curves in \cite{KrSh20} in an alternative manner. The primary focus of their construction is to establish the Biswas-Iyer-Simpson correspondences for parabolic $\lambda$-connections and to develop the parabolic nonabelian Hodge theory. Our approach differs from that of Krishnamoorthy and Sheng, as we directly construct the parabolic inverse Cartier functor $\mathcal C_n^{-1}$ which coincides with the parabolic inverse Cartier functor constructed in \cite{KrSh20} when $n=1$ (\autoref{thm_CompareInverseCartierFunctor}), without the necessity of descending it form a cover. We  extend \cite[Theorem 5.3]{LSZ19} to the parabolic setting.

\begin{theorem}[\autoref{thm_equFunctorHdRF&FFMod}]
We keep the same notation as in the above theorem.
Then there is an equivalence of categories from the category of $f$-periodic parabolic Higgs-de Rham flows over $(Y_n,D_{Y_n})/S_n$ to the category of strict $p^n$-torsion parabolic Fontaine-Faltings modules with $W_n(\mathbb F_{p^f})$-endomorphism structure over $(Y_n,D_{Y_n})/S_n$.
\end{theorem}

As an application of the theory of parabolic Higgs-De Rham flows, We also study periodic parabolic Higgs bundles over a marked projective line. Let $M_{\Hig}^N$ be the moduli space of semistable graded parabolic Higgs bundles of rank $2$ degree $0$ over $\mathbb{P}^1_k$, such that the Higgs fields have $m$ poles $\{x_1, x_2, \ldots, x_m\}$ and the parabolic weights are in the set $$J_N\coloneqq\left\{\frac{k}{N} : k\in [1,N-1]\cap \mathbb Z, \gcd(k, N) = 1\right\}$$ at $x_1$ and $0$ at other points. Let $(E, \theta)$ be a parabolic Higgs bundle over $\mathbb{P}^1_k$ such that the isomorphism class of $(E, \theta)$ is in $M_{\Hig}^N(k)$. The selfmap is defined on the set $M_{\Hig}^N(k)$, which is the composition of the grading functor followed by the parabolic inverse Cartier functor $\mathcal C_1^{-1}$ (see \autoref{sec_SelfmapParaHiggsBundle}). Fixing a maximal irreducible component $M_{\Hig}^{N, \frac{d}{N}}$, we can prove the following theorem:
\begin{theorem}[\autoref{Thm_SelfmapDominant}] For any $\frac{d}{N}\in J_N$, let
\[ d' = \min\left\{N \left\langle \frac{pd}{N} \right\rangle, N \left\langle \frac{p(N-d)}{N} \right\rangle\right\},\]
then the selfmap induces a rational and dominant map $\varphi: M_{\Hig}^{N, \frac{d}{N}} \dashrightarrow M_{\Hig}^{N, \frac{d'}{N}}$.
\end{theorem}

When $d = d'$, with the same method as in \cite[Theorem 4.4]{SYZ22}, we can prove the following corollary which enables us to construct infinitely many parabolic Fontaine-Faltings modules, hence infinitely many parabolic crystalline representations of $\pi_1^{\et}(\mathbb P_K^1-\{0,1,\lambda,\infty\})$.

\begin{corollary}[\autoref{thm_PeriodicPointsDense}]
For any $\frac{d}{N}\in J_N$, let $k$ be the smallest positive integer such that
\[\left\{\left\langle\frac{p^kd}{N}\right\rangle, \left\langle\frac{p^k(N-d)}{N}\right\rangle\right\}=\left\{\left\langle\frac{d}{N}\right\rangle, \left\langle\frac{N-d}{N}\right\rangle\right\}.\]
 the set of periodic points of $\varphi^k$ is Zariski dense in $M_{\Hig}^{N, \frac{d}{N}}$.
\end{corollary}

In 2018, motivated by Kontsevich’s program on Langlands correspondence, Kang Zuo proposed the concepts \emph{arithmetic Higgs bundles}. Let $(Y,D_Y)$ be a log smooth pair over a number field $K$. A Higgs bundle $(E,\theta)$ over $(Y,D_Y)/K$ is \emph{arithmetic} if  there is a positive integer $f$ such that for almost all finite place $\mathfrak p$ of $K$, the modulo $ \mathfrak p$ reduction of $(E,\theta)$ is a term of a periodic Higgs-de Rham flow whose period $\leq f$. Kang Zuo conjectured that

\begin{conjecture}[Zuo 2018]
    A Higgs bundle over $(Y,D_Y)/K$ is arithmetic if and only if it is motivic.
\end{conjecture}

Building on this framework, Krishnamoorthy and Sheng later introduced the notion of periodic Higgs bundles on complex curve in \cite{KrSh20}. A Higgs bundles $(E,\theta)$ over a complex curve $C$ is periodic if it has a spreading out $\mathscr{(E,\theta)}$ over a finite generated $\mathbb Z$-subalgebra of $\mathbb C$ such that  $ \mathscr{(E,\theta)}\otimes \Spec \overline{\kappa (\mathfrak p)}  $ is a term of a periodic Higgs-de Rham flow of bounded periodicity for almost all closed points $\mathfrak p \in \Spec A$. They proved that motivicity implies periodicity (See \cite[Theorem 1.3]{KrSh20}).

Subsequently, in \cite{LaLi23}, Lam and Litt proved a conjecture of Sun, Yang and Zuo (\cite[Conjecture 4.8]{SYZ22}),
which corresponds a class of periodic parabolic Higgs bundles on $\mathbb P^1_{\mathbb C}$ with four marked points $\{0,1,\lambda,\infty\}$ to the $p$-torsion points of the elliptic curve
\[C_\lambda: y^2=x(x-1)(x-\lambda).\]
Combining with a theorem of Yang and Zuo in \cite{YaZu23a}, one obtained a criterion for the motivicity of a class of parabolic Higgs bundles on $\mathbb P^1_{\mathbb C}$.

\begin{theorem}[{\cite[Theorem 1.6]{YaZu23a},\cite[Theorem 1.2.2]{LaLi23}}]\label{thm_Motivicity}
A semistable parabolic graded Higgs bundle of degree $0$ on the $\mathbb P^1_{\mathbb C}$ with $4$-punctures ${0,1,\lambda,\infty}$ whose parabolic weights are $\frac{1}{2}$ at $\infty$ and $0$ at $0,1,\lambda$ if and only if
the zero of the Higgs field  is the $x$-coordinate of a  torsion point in $C_\lambda$.
\end{theorem}

Let $K$ be a number field, Denote $M_{\Hig,K}^{N,\frac{d}{N}}$ the moduli space of semistable graded parabolic Higgs bundles of rank $2$ degree 0 over $\mathbb{P}^1_K$ such that the Higgs fields have $m$ poles $\{x_1, x_2, \ldots, x_m\}$ ($m\geq 4$) and the parabolic weights are in the set $J_N$ at $x_1$ and 0 at other points. Then it is natural to ask the following questions:
\begin{question}
    \begin{enumerate}
        \item How to determine the periodic points in $M_{\Hig}^{N,\frac{d}{N}}$ for any $\frac{d}{N}\in J_N$?
        \item How to determine the arithmetic points in $M_{\Hig,K}^{N,\frac{d}{N}}$ for any $\frac{d}{N}\in J_N$?
        \item Is there a criterion for the motivicity of a  semistable parabolic graded Higgs bundle of degree $0$ on the $\mathbb P^1_{\mathbb C}$ with $m$-punctures ($m\geq 4$) whose parabolic weights are contained in $J_N$ at $x_1$ and $0$ at other points,  that is analogous to \autoref{thm_Motivicity}?
        \item More generally, how to determine the periodicity,  arithmeticity and motivicity of Higgs bundles on any log smooth pair?
    \end{enumerate}
\end{question}

This paper is organized as follows. In \autoref{sec_main_para}, we review some basic facts about parabolic bundles. Inspired by the work of Iyer and Simpson \cite{IySi07}, we introduce the notion of pullback for parabolic $\lambda$-flat bundles. Through comparison with the pullback functor in \cite{KrSh20}, we show that this notion generalizes the functor constructed therein (see \autoref{thm_ComparePullbackDeRhamBundle}).

\autoref{sec_Para_FF} is dedicated to establishing the theory of parabolic Fontaine-Faltings modules. We first review the work of Faltings on the theory of Fontaine-Faltings modules. Subsequently, we construct the parabolic Faltings' tilde functor $\widetilde{(\cdot)}$ and the parabolic Frobenius pullback functor $\mathcal F_n$ over the truncated ring $W_n(k)$. Then we introduce the notions of parabolic Fontaine-Faltings modules and construct the $\mathbb D^{par}$-functor for parabolic strict $p^n$-torsion Fontaine-Faltings modules (\autoref{thm_ParabolicDFunctor}). The objects in the essential image of $\mathbb D^{par}$ are called   parabolic crystalline representations.

In \autoref{sec_Para_Higgs_de_Rham}, we generalizes the theory of  Higgs-de Rham flows in \cite{LSZ19} to parabolic setting. We first review the work of Lan, Sheng, and Zuo on the theory of Higgs-de Rham flows. Subsequently, we construct the  parabolic inverse Cartier functor and as an application, we introduce the notions of parabolic Higgs-de Rham flows and establishing an equivalence of categories between the category of $f$-periodic parabolic Higgs-de Rham flows and the category of parabolic strict $p^n$-torsion Fontaine-Faltings modules with $W_n(\mathbb F_{p^f})$-endomorphism structure(\autoref{thm_equFunctorHdRF&FFMod}). In the case $n=1$, our construction coincides with the work of Krishnamoorthy and Sheng in \cite{KrSh20} (see \autoref{thm_CompareInverseCartierFunctor}).

In \autoref{sec_App}, we apply the theory of parabolic Higgs-de Rham flows to study a special selfmap. This selfmap acts on the set of $k$-points of the moduli space $M_{\Hig}^N$. The moduli space $M_{\Hig}^N$ classifies semistable, rank-$2$, degree $0$ parabolic Higgs bundles on an $m$-marked projective line $\mathbb{P}_k^1$ with non-zero parabolic weights at a single point, and these weights are located in $(0,1) \cap \frac{1}{N}\mathbb{Z}$. \autoref{thm_PeriodicPointsDense} implies that the periodic points are Zariski dense in the moduli space, enabling the construction of infinitely many parabolic crystalline representations. When $m=4$, we enumerate all families of elliptic curves with four singular fibers, which induce 1-periodic parabolic Higgs bundles on $\mathbb{P}_k^1$ with four marked points, as detailed in \autoref{Table1}.

\subsection*{Acknowledgments}
J.Y. and Z.L. acknowledge the financial support from the National Natural Science Foundation of China (Grant No. 12201595), the Fundamental Research Funds for the Central Universities, and the CAS Project for Young Scientists in Basic Research (Grant No. YSBR-032). K.Z. and J.Y. gratefully acknowledge the financial support from the Key Program (Grant No. 12331002) and the International Collaboration Fund (Grant No. W2441003) of the National Natural Science Foundation of China.


\section{\bf Parabolic Structures} \label{sec_main_para}
In this section, we undertake a study of de Rham bundles and Higgs bundles equipped with parabolic structures. We adopt the terminology and notation for parabolic structures established in \cite{IySi07}; see also \cite{KrSh20}.

\subsection{Introduction to parabolic vector bundles}

Our intention is not to introduce parabolic objects on arbitrary spaces. Although the methods and results remain applicable to a broader range of scenarios, for the sake of brevity, we restrict our focus to the following specific spaces $(Y, D_Y)/S$.

Let $p$ be an odd prime number, and let $S$ be a connected Noetherian (formal) scheme. For a smooth separated scheme $h:Y\to S$ (or a smooth separated formal scheme over $S$ if $S$ is a formal scheme), let $D_Y\subseteq Y$ be a relative normal crossing divisor with irreducible irreducible components $D_{Y, i}$, $i=1, 2, \cdots, n$. We denote $U_Y := Y - D_Y$ and by $j_Y$ the open immersion $j_Y\colon U_Y \hookrightarrow Y$.

We set $\Omega_{Y/S}^1$ to be the sheaf of relative $1$-forms and $\Omega_{Y/S}^1(\log D_Y)$ to be the sheaf of relative $1$-forms with logarithmic poles along $D_Y$. By the smoothness of $Y$ over $S$, both of these sheaves are vector bundles over $Y$.

\subsubsection{Parabolic vector bundles}

In this section, we introduce the concept of parabolic vector bundles, which is based on \cite{IySi07} and \cite{Mar92}.

When $S=\Spec K$ with $K$ a field, we recall in \cite[Section 2]{IySi07} that a parabolic sheaf $F$ on $(Y, D_Y)/K$ is a collection of torsion-free sheaves $F_{\alpha}$ indexed by multi-indices $\alpha=(\alpha_1,\alpha_2,\cdots,\alpha_n)\in \mathbb Q^n$  with inclusions
of sheaves of $\mathcal O_X$-modules $F_\alpha\hookrightarrow F_\beta $ whenever $\alpha_i\leq \beta_i$ for $1\leq i\leq n$ (a condition which we write as $\alpha\leq \beta$ in what follows), subject
to the following hypotheses:
\begin{itemize}
\item (Normalization) Let $\delta^i=(0, \cdots, 1, \cdots, 0)\in\mathbb{Q}^n$, then $F_{\alpha+\delta^i} = F_\alpha(D_{Y, i})$ (compatibly with the inclusion);
\item (Semicontinuity) For any given $\alpha$, there exists a constant $c=c(\alpha)>0$ such that $F_{\alpha+\epsilon} = F_\alpha$ for $\epsilon=(\epsilon_1, \cdots, \epsilon_n)\in\mathbb{Q}^n$ with $0 \leq \epsilon_i \leq c$.
\item (Rationality) For any given $\alpha$, $\inf\{\beta_i:F_\beta=F_\alpha\}$ is rational for all $i$.\footnote{We add this condition to make sure that all parabolic weights are in $\mathbb Q$.}
\end{itemize}
The set of \emph{parabolic weights}  of a parabolic sheaf $F$ along $D_{Y,i}$ is the set of rational numbers $\alpha\in [0,1)$ such that
 $$F_{-(\alpha+\epsilon)\delta^i}\neq F_{-\alpha\delta^i}$$ for any $\epsilon>0$. A \emph{morphism between two parabolic sheaves} from $F$ to $F'$ is a collection of compatible morphisms of sheaves $f_\alpha\colon F_\alpha \rightarrow F'_\alpha$.
\begin{remark}
\begin{enumerate}
\item[1).] The first hypothesis implies that the quotient sheaves $F_\alpha/F_\beta$ for $\beta \leq \alpha$ are supported at $D_Y$.
\item[2).] The second and third hypotheses imply that the structure is determined by the sheaves $F_\alpha$ for a finite collection of rational indices $\alpha$ with $0 \leq \alpha_i < 1$.
\item[3).] The extension $j_{Y*}\left( j_Y^* F_\beta\right)$ does not depend on the choice of $\beta \in \mathbb{Q}^n$, denoted by $F_\infty$. Note that the $F_\alpha$ may all be considered as subsheaves of $F_\infty$.
\end{enumerate}
\end{remark}

Let $F$ and $F'$ be two parabolic sheaves on $(Y,D_Y)/K$, in  \cite[page 352]{IySi07},  Iyer-Simpson defined the \emph{tensor product} $F\otimes F'$ of $F$ and $F'$ by setting
\[(F \otimes F')_{\alpha} := \sum_{\beta + \gamma = \alpha} F_\beta \otimes F_\gamma \subset F_\infty \otimes F'_\infty.\]
Then $F \otimes F' := \{(F \otimes F')_{\alpha}\}$ forms a parabolic sheaf.

More generally, when $S$ is not the spectrum of a field, Maruyama-Yokogaw introduced the notion of \emph{flat family of paraboilc sheaves} in \cite[page 7]{Mar92}, which we refer as \emph{parabolic sheaf on $(Y,D_Y)/S$} in this paper as follows.
\begin{definition}
    A \emph{parabolic sheaf $F$ on $(Y,D_Y)/S$} is  is a collection of \emph{$h$-torsion-free}
\footnote{A coherent $\mathcal O_Y$-module $F$ is $h$-torsion-free if it is $S$-flat, and for every geometric fiber $Y_s$ of $h$, $F\otimes \mathcal O_{Y_s}$ is a torsion-free $\mathcal O_{Y_s}$-module. See also \cite[Definition 1.13]{Mar92}.}
sheaves $F_{\alpha}$ indexed by multi-indices $\alpha=(\alpha_1,\alpha_2,\cdots,\alpha_n)\in \mathbb Q^n$  with inclusions
of sheaves of $\mathcal O_X$-modules $F_\alpha\hookrightarrow F_\beta $ whose cokernel is  $S$-flat, whenever $\alpha\leq \beta$, subject to the same three hypotheses as when $S=\Spec K$.
\end{definition}
\begin{remark}
     The conceptions of parabolic weights, morphisms between parabolic sheaves, tensor product can be extended to the general case without difficulty.
\end{remark}
For any $\alpha = (\alpha_1, \cdots, \alpha_n) \in \mathbb{Q}^n$, denote
\[\alpha D_Y = \alpha_1D_{Y, 1} + \cdots + \alpha_n D_{Y, n}\]
which is a rational divisor supported on $D_Y$. All rational divisor supported on $D_Y$ are of this form. Denote
\[\lfloor\alpha\rfloor \coloneqq (\lfloor\alpha_1\rfloor, \lfloor\alpha_2\rfloor, \cdots, \lfloor\alpha_n\rfloor)\]
where $\lfloor\alpha_i\rfloor$ is the maximal integer smaller than or equal to $\alpha_i$. In particular, $\lfloor\alpha\rfloor D_Y$ is an integral divisor supported on $D_Y$.

\begin{example}[Trivial Parabolic Structure]\label{TriParaBun} For any $h$-torsion-free sheaf $E$ on $Y$, it may be considered as a parabolic sheaf (we say ``with trivial parabolic structure'') by setting
\[E_\alpha = E(\lfloor\alpha\rfloor D).\]
Sometimes, we will identify the usual sheaf $E$ with the parabolic sheaf $\{E_\alpha\}$ defined in this way.
\end{example}

\begin{example}\label{exa:paraLineBund} Let $L$ be a line bundle and let $\gamma=(\gamma_1, \cdots, \gamma_n) \in \mathbb{Q}^n$ be a rational multi-index. Then $\{L(\lfloor\alpha + \gamma\rfloor D_Y)\}_\alpha$ forms a parabolic sheaf, which is denoted by $L(\gamma D_Y)$.
Clearly, for any two line bundles $L, L'$ and two multi-indices $\gamma, \gamma'\in\mathbb{Q}^n$, one has
\[L(\gamma D)\otimes L'(\gamma' D') = (L\otimes L')\Big((\gamma +\gamma') D\Big).\]
Then the set of all isomorphic classes of parabolic line bundles forms an abelian group under the tensor product, which contains the Picard group of $Y$ as a subgroup in the natural way.
\end{example}

\begin{definition}[{\cite[Definition 2.2]{IySi07}}]
\begin{itemize}
\item[1).] The parabolic sheaves appearing in \autoref{exa:paraLineBund} are called \emph{parabolic line bundles}.
\item[2).] A \emph{parabolic vector bundle} is a parabolic sheaf which is locally isomorphic to a direct sum of parabolic line bundles. (Simpson called this a \emph{locally abelian parabolic vector bundle}.)
A \emph{morphism between two parabolic vector bundles} is a morphism between their underlying parabolic sheaves.
\end{itemize}
\end{definition}

\subsubsection{Degrees and semistability}
In order to define the degree and the semistability of a parabolic vector bundle, we assume that $h:Y\to S$ is projective and fix a $h$-relative ample divisor $\mathcal O(1)$ on $Y$.

Let $F$ be a coherent sheaf on $Y$ which is flat over $S$. By the local constancy of Hilbert polynomials(see \cite[III 9.9]{Har77}), the Hilbert polynomial of the coherent sheaf $F_s$ on $Y_s$ for any points $s\in S$ does not depend on the choice of the point $s$.
It is well-defined to set
\[\deg(F)\coloneqq \deg (F_s).\footnote{Let $\alpha_d(\mathcal O_{Y_s})$ and $ \alpha_d(F_s)$ be the leading coefficients of the Hilbert polynomials of $\mathcal O_{Y_s}$ and $F_s$, then the degree of $F_s$ is defined to be $\alpha_d(F_s)/\alpha_d(\mathcal O_{Y_s})$. See also \cite[Definition 1.2.2]{HuLe10}.}\]
We define the degree of a parabolic vector bundle as follows.
\begin{definition} Let $F$ be a parabolic sheaf over $(Y, D_Y)/S$.
\begin{itemize}
\item For any $t\in \mathbb R$, set $$\deg_t (F)\coloneqq\inf \{\deg (F_{\alpha})\mid \alpha\geq (t,t,\cdots,t)\}.$$ It is a piecewise constant function of $t$ with finite jump sets in $[0,1)$.
The \emph{degree of $F$} is defined as
\[\deg(F) \coloneqq \int_0^1 \deg_t(F)\rmd t.\]
\item The parabolic vector bundle $F$ is called \emph{semistable} (resp. \emph{stable}) if for any proper sub parabolic sheaf $F'\subsetneq F$, one has
\[\frac{\deg(F')}{\rank(F')} \leq \frac{\deg(F)}{\rank(F)} \qquad \text{\Big(resp.} \frac{\deg(F')}{\rank(F')} < \frac{\deg(F)}{\rank(F)} \text{\Big)}.\]
\end{itemize}
\end{definition}

\begin{example}
    We recall the  parabolic line bundle $\deg L(\gamma D)$ in \autoref{exa:paraLineBund}. Its degree
    \[\deg L(\gamma D)=\deg L + \sum_{i=1}^n \gamma_i.\]
\end{example}

\subsubsection{The pullback of parabolic vector bundles}
 We recall a definition proposed by Simpson in \cite[Section 2.2]{IySi07} for the pullback of parabolic vector bundles. In the following, we will give an easy-to-use format description of the pullback, see \autoref{prop:PullbackParabolicBundles}. Roughly speaking, the parabolic pullback can be globally described by the pullback for usual vector bundles and the pullback for parabolic line bundles.

 \begin{setup}
     Let $f\colon (Y', D_{Y'}) \rightarrow (Y, D_Y)$ be a morphism of smooth separated schemes with relative normal crossings divisors over $S$ such that $f^{-1}(D_Y)\subset D_{Y'}$.
 \end{setup}

\begin{definition} \label{def_pullbackbundle}
\begin{enumerate}
\item[(1)] If $F$ is a parabolic line bundle, then there exists a line bundle $L$ and a rational divisor $B$ supported on $D$ such that $F=L(B)$. We define
\[f^*F \coloneqq (f^*L)(f^*B).\]
\item[(2)] For a general parabolic vector bundle $V$, we define $f^*V$ through localization, thereby reducing the case to parabolic line bundles.
\end{enumerate}
\end{definition}

Let $V=\{V_\alpha\}$ be a parabolic vector bundle over $(Y, D_Y)/S$. For any $\gamma\in\mathbb{Q}^n$, we identify the two parabolic vector bundles
\[V = V(-\gamma D)\otimes\mathcal{O}_Y(\gamma D).\]
Clearly, once we identify the usual vector bundle with its associated parabolic vector bundle with the trivial parabolic structure, then $V_0$ can be viewed as a parabolic subsheaf of $V$
\[V_0 \subseteq V.\]
In particular, for the parabolic vector bundle $V(-\gamma D)$, we have a parabolic subsheaf
\[V(-\gamma D)_0\hookrightarrow V(-\gamma D).\]
By pullback along $f$ and tensoring $f^*\Big(\mathcal{O}_Y(\gamma D)\Big)$, we obtain a parabolic subsheaf
\[f^*_{\gamma}(V)\coloneqq f^*\Big(V(-\gamma D)_0\Big)\otimes f^*\Big(\mathcal{O}_Y(\gamma D)\Big) \subseteq f^*\Big(V(-\gamma D)\Big)\otimes f^*\Big(\mathcal{O}_Y(\gamma D)\Big) = f^*(V).\]
\begin{proposition} \label{prop:PullbackParabolicBundles}
$f^*V = \sum\limits_{\gamma\in\mathbb{Q}^n} f^*_{\gamma}(V)$.
\end{proposition}

\begin{proof}
By localization, we reduce it to the case of parabolic line bundles, which follows by taking $\mathcal{L}$ to be the parabolic line bundle itself.
\end{proof}

In \cite{Bis97}, Biswas constructed the pullback functor of parabolic vector bundles in the case that $f$ is a finite Galois morphism. In this subsection, we compare Biswas' construction with the pullback functor defined in \autoref{def_pullbackbundle} and point out that the pullback functor in \autoref{def_pullbackbundle} is a generalization of it. In \cite{Bis97}, the parabolic pullback functor is constructed under the following assumptions.
\begin{setup}\label{setup_PullBackPar}
Let $k$ be a field and $S=\Spec k$. Let $N$ be a positive integer coprime to $\Char k$. Let $f:Y'\to Y$ be a finite Galois morphism with Galois group $G$ such that $D_{Y'}=(f^*D_{Y})_{red}$ is normal crossing and $f^*D_{Y, i}=k_i N D_{Y', i}$ for $1\leq i\leq n$, where $k_i\in \mathbb{Z}_+$ and $D_{Y', i}=(f^*D_{Y, i})_{red}$.
\end{setup}
Given a parabolic vector bundle $V$ over $(Y, D_Y)/S$ whose weights are
contained in $\frac{1}{N}\mathbb{Z}$, set $E_i^0=f^*(V_0\otimes \mathcal{O}(D_Y))$ for $1\leq i\leq n$. Let
\[0\leq \alpha_i^1< \alpha_i^2< \ldots\alpha_i^{w_i}<1\]
be the parabolic weights supported at $D_{Y, i}$. Set
\[\alpha_{i}^{w_i+1}=1, \quad m_i^j=N\alpha_{i}^j.\]
 For each $i$, we construct a filtration of bundles
\[E_i^0\supseteq E_i^1\supseteq\ldots \supseteq E_i^{w_i}\]
inductively for $1\leq j\leq w_i$ by
\[E_i^{j+1}\coloneqq\ker\left(E_i^j\to f^* \left(
V_{-\alpha_i^{j+1}\delta^i}/V_{-\alpha_i^{j+2}\delta^i}\otimes \mathcal O_{Y}(D_Y)
\right)\big|_{k_i(N-m_i^{j+1})D_{Y', i}}
\right).\]
Define
\[f_{par}^*V\coloneqq \bigcap_{i=1}^n E_{i}^{w_i}.\]

\begin{proposition}\label{Thm_PullBack}
Let $V$ be a parabolic vector bundle over $(Y, D_Y)/S$ with weights contained in $\frac{1}{N}\mathbb{Z}$. Then $f^*V = f_{par}^*V$ and has trivial parabolic structure.
\end{proposition}
\begin{proof}
By localization, we reduce to the case of a parabolic line bundle $\mathcal{L}=\mathcal{O}_Y(\gamma D_Y)$ with $\gamma=(\gamma_1, \ldots, \gamma_n)\in ([0,1)\cap \mathbb{Q})^n$, the parabolic weight of $\mathcal L$ supported at $D_{Y, i}$ is $\gamma_i$. For any  $1\leq i\leq n$, one has
\begin{equation*}
\begin{split}
E_i^0&=\mathcal{O}_{Y'}\left(\sum_{i=1}^n k_i N D_{Y'_i}\right),\\
E_i^1&=\ker\left(E_i^0\to f^*\Big(\mathcal O_Y/ \mathcal O_Y(-D_{Y,i})\otimes \mathcal O_{Y}(D_Y)\Big)\Big|_{k_iN(1-\gamma_i)D_{Y',i}}
\right)\\
&=\mathcal{O}_{Y'}\left(\sum_{j\neq i}k_j N D_{Y', j} +k_iN \gamma_iD_{Y', i}\right).
\end{split}
\end{equation*}
Then as sub sheaves of $f^*\mathcal L_0(D_{Y})$,
$f_{par}^*\mathcal{L}=\cap_{i=1}^n E_{i}^1=\mathcal{O}_{Y'}(\sum_{i=1}^n k_i N\gamma_i D_{Y',i})=f^*\mathcal{L}$.
\end{proof}

\subsection{Parabolic $\lambda$-flat bundles}
In this subsection, we recall the definition of parabolic $\lambda$-connection, see also \cite[Section 4]{Sim97}.
\subsubsection{Logarithmic $\lambda$-connections}
A \emph{(logarithmic) $\lambda$-connection} on a sheaf $V$ of $\mathcal{O}_Y$-modules over $(Y, D_Y)/S$ is an $\mathcal{O}_S$-linear map $\nabla\colon V\rightarrow V\otimes_{\mathcal{O}_Y} \Omega_{Y/S}^1(\log D_Y)$ satisfying the Leibniz rule $\nabla(rv) = v\otimes \rmd r + \lambda r\nabla(v)$ for any local section $r\in \mathcal{O}_Y$ and $v\in V$. Given a (logarithmic) $\lambda$-connection, there are canonical maps
\[\nabla\colon V\otimes\Omega_{Y/S}^i(\log D_Y) \rightarrow V\otimes\Omega_{Y/S}^{i+1}(\log D_Y)\]
given by $s\otimes \omega \mapsto\nabla(s)\wedge \omega +\lambda s \otimes \rmd\omega$.
The curvature $\nabla\circ\nabla$, the composition of the first two canonical maps, is $\mathcal{O}_Y$-linear and contained in $\mathcal{E}nd(V)\otimes_{\mathcal{O}_Y} \Omega^2_{Y/S}(\log D_Y)$.
The $\lambda$-connection is called \emph{integrable} if the curvature vanishes.
A pair $(V,\nabla)$ which is made up of a $\mathcal O_Y$-module $V$ and an integrable $\lambda$-connection $\nabla$ on $V$ is called a \emph{(logarithmic) $\lambda$-flat sheaf}. A  (logarithmic) $\lambda$-flat sheaf is called a \emph{(logarithmic) de Rham sheaf} (reps. \emph{(logarithmic) Higgs sheaf}) if  $\lambda=1$(reps. $\lambda=0$).

For a \emph{(logarithmic) $\lambda$-flat sheaf} on $(Y,D_Y)/S$, one has a natural complex:
\[ 0\rightarrow V
\xrightarrow{\nabla} V\otimes\Omega^1_{Y/S}(\log D_Y)
\xrightarrow{\nabla} V\otimes\Omega^2_{Y/S}(\log D_Y)
\xrightarrow{\nabla} V\otimes\Omega^3_{Y/S}(\log D_Y) \rightarrow \cdots\]
A logarithmic $\lambda$-flat sheaf (reps. logarithmic de Rham sheaf, logarithmic Higgs sheaf) is called a \emph{logarithmic $\lambda$-flat bundle} (reps. \emph{logarithmic de Rham bundle}, \emph{logarithmic Higgs bundle}), if the underlying sheaf is a vector bundle over $Y$. Denote by \emph{$\MIC(Y, D_Y)$} the category of all logarithmic de Rham bundles over $(Y, D_Y)$.
We recall the following well-known result.
\begin{lemma}\label{thm_NatExtLogConn} Let $(V, \nabla)$ be a logarithmic $\lambda$-flat bundle over $(Y, D_Y)/S$. Denote by $j_Y$ the open immersion of $U=Y\setminus D_Y$ into $Y$. Then
\begin{enumerate}
\item[(1)] the $\lambda$-connection $\nabla$ on $V$ can be uniquely extended onto $V_\infty := j_{Y*}j^*_Y(V)$ under the natural injection $V\hookrightarrow V_\infty$, and
\item[(2)] the extension $\lambda$-connection can be restricted to be $\nabla'$ onto $V(D')$ for any integral divisor $D'=\sum_{i=1}^{n}a_{i}D_{Y, i}$ supported on $D$.
In particular, if $D'$ is also positive, then the logarithmic $\lambda$-flat bundle $(V, \nabla)$ extends to another one $(V(D'), \nabla')$ with natural injection
\[(V, \nabla) \hookrightarrow (V(D'), \nabla').\]
\item[(3)] Let $\nabla''$ be a $\lambda$-connection onto $V$ such that $j_Y^*\nabla = j_Y^*\nabla''$, then $\nabla''=\nabla$.
\end{enumerate}
\end{lemma}
\begin{proof}
(1) and (2) follow from \cite[lemma 2.7]{EsVi92}. For (3), we just need to check it locally. Let $\{t_1,t_2,\ldots,t_d\}$ be a local \'etale coordinate of $Y/S$ and $D_Y$ is defined by  $\{t_1,t_2,\ldots,t_n\}$. Write $t^\alpha=t_1^{\alpha_1}t_2^{\alpha_2}\ldots t_n^{\alpha_n}$ for short.
For any local section $s$ of $V$, there exist $\alpha\in \mathbb N^n$ such that $s=t^{\alpha}s'$ with $s'$ a local section of $j_Y^* V$. then
\[\nabla''(s)=t^{\alpha}\nabla''(s')+\lambda s'\otimes \rmd(t^{\alpha})=t^{\alpha}\nabla(s')+\lambda s'\otimes \rmd(t^{\alpha})=\nabla(s).\qedhere\]
\end{proof}

\subsubsection{Parabolic $\lambda$-flat sheaves}
\begin{definition}
Let $V$ be a parabolic sheaf on $(Y,D_Y)/S$.
\begin{enumerate}
    \item A \emph{parabolic $\lambda$-connection} on $V$ is a family of logarithmic  $\lambda$-connections $$\{\nabla_\alpha:V_\alpha \to V_\alpha\otimes \Omega_{Y/S}^1(\log D_Y)\}_{\alpha\in \mathbb Q^n}$$ which is compatible with the inclusion $V_\alpha\to V_{\beta}$ when $\beta\geq \alpha$. A parabolic $\lambda$-connection $\{\nabla_\alpha\}_{\alpha\in \mathbb Q^n}$ is integrable if for any $\alpha\in \mathbb Q^n$, $\nabla_\alpha$ is integrable.
    \item A parabolic $\lambda$-connection is called a  \emph{parabolic connection} (reps. \emph{parabolic Higgs field}) when $\lambda=1$(reps. $\lambda=0$).
    \item A pair $(V,\nabla)$ consisting of a parabolic vector bundle $V$ on $(Y,D_Y)/S$ and an integrable parabolic $\lambda$-connection $\nabla$(reps. parabolic connection, parabolic Higgs field) on $V$ is called a \emph{parabolic $\lambda$-flat bundle} (reps.  \emph{parabolic de Rham bundle}, \emph{parabolic Higgs bundle}).
    \item A parabolic flat $\lambda$-flat bundle $(V,\nabla)$ is semistable (reps. stable) if for any proper $\nabla$-invariant parabolic subbundle $W\subseteq V$, one has
    \[\frac{\deg(W)}{\rank(W)} \leq \frac{\deg(V)}{\rank(V)} \qquad \text{\Big(resp.} \frac{\deg(W)}{\rank(W)} < \frac{\deg(V)}{\rank(V)} \text{\Big)}.\]
\end{enumerate}
\end{definition}

\begin{remark}
\begin{enumerate}
\item[(1).] The $(V_\alpha, \nabla_\alpha)$ may be considered as sub $\lambda$-flat sheaf of
\[(V_\infty, \nabla_\infty) := \varinjlim_\beta (V_\beta, \nabla_\beta) = j_{Y*}j^*_Y (V_\alpha, \nabla_\alpha).\]
\item[(2).] Tensor product of two parabolic $\lambda$-flat bundles can be naturally defined: for any two parabolic $\lambda$-flat bundles $(V, \nabla)$ and $(V', \nabla')$,  the underlying parabolic vector bundle of their tensor product is $V\otimes V'$ and the $\lambda$-parabolic connection is defined as the restriction of the $\lambda$-connection $\Big(\nabla_\infty\otimes \id + \id\otimes\nabla'_\infty\Big)$ on $V_\infty\otimes V'_\infty$.
\end{enumerate}
\end{remark}

\begin{example}
Let $\gamma\in\mathbb{Q}^n$ be a multiple index and let $(V,\nabla)$ be a logarithmic $\lambda$-flat bundle on $(Y, D_Y)/S$. There is a natural parabolic $\lambda$-connection, denoted by $\nabla(\gamma D)$, on the parabolic vector bundle $V(\gamma D)$ given by
\[\nabla(\gamma D)_\alpha:=\nabla_\infty \mid_{V(\gamma D)_{\alpha}}.\]
The logarithmic $\lambda$-connection $\nabla$ may be considered as a parabolic $\lambda$-connection (we say ``with trivial parabolic structure'') by identifying it with $\nabla(0D)$.
Assume $(V, \nabla)=(\mathcal{O}_Y, \lambda\rmd)$. We call parabolic $\lambda$-flat line bundles of the form $(\mathcal{O}_Y(\gamma D), \lambda\rmd(\gamma D))$ \emph{shifting parabolic $\lambda$-flat line bundles}.
\end{example}

As a consequence of \autoref{thm_NatExtLogConn}, one has the following uniqueness result for parabolic $\lambda$-connections.
\begin{lemma}\label{thm_UnqiuePara}
Let $V = \{V_\alpha\}_{\alpha\in \mathbb Q^n}$ be a parabolic vector bundle over $(Y, D_Y)/S$.
\begin{enumerate}
\item Suppose $\nabla_0$ is a logarithmic $\lambda$-connection on $V_0$.
Then for any $\alpha\neq 0$, there exists a unique parabolic $\lambda$-connection $\nabla_\alpha$ on $V_\alpha$ such that $\nabla=\{\nabla_\alpha\}_\alpha$ forms a parabolic $\lambda$-connection on $V$.
\item Suppose $\nabla=\{\nabla_\alpha\}_\alpha$ and $\nabla'=\{\nabla'_\alpha\}_\alpha$ are two parabolic $\lambda$-connections on $V$. If $j^*_Y\nabla_\alpha=j^*_Y\nabla'_\beta$ for some $\alpha, \beta\in \mathbb{Q}^n$, then $\nabla=\nabla'$.
\end{enumerate}
\end{lemma}

\subsubsection{Parabolic $\lambda$-flat line bundles}

By tensoring with some shifting parabolic $\lambda$-flat line bundle, a parabolic $\lambda$-flat line bundle can be modified to have a trivial parabolic structure. Thus, we have the following result.
\begin{lemma}
Let $(\mathcal{L}, \nabla)$ be a parabolic $\lambda$-flat line bundle over $(Y, D_Y)/S$. Denote by $w_i\in [0, 1)\cap \mathbb{Q}$ the parabolic weight of $\mathcal{L}$ along $D_{Y, i}$ and denote $w=(w_1, \cdots, w_n)$. Then
\begin{enumerate}[$(1)$]
\item we have a decomposition of $(\mathcal{L}, \nabla)$ into a tensor product of a usual logarithmic $\lambda$-flat bundle and a shifting parabolic $\lambda$-flat line bundle
\[(\mathcal{L}, \nabla) = (\mathcal{L}_0, \nabla_0) \otimes (\mathcal{O}_Y(w D), \rmd(w D)).\]
\item Suppose we have another such decomposition:
\[(\mathcal{L}, \nabla) = (L, \nabla)\otimes (\mathcal{O}_Y(\gamma D), \rmd(\gamma D)),\]
then $\omega - \gamma \in \mathbb{Z}^n$ and
\[(L, \nabla)=(\mathcal{L}_0, \nabla_0)\otimes (\mathcal{O}_Y((w-\gamma)D), \rmd((w-\gamma)D)).\]
\end{enumerate}
\end{lemma}

\begin{proof}
Consider $(\Delta V, \Delta \nabla):= (\mathcal{L}, \nabla) \otimes (\mathcal{L}_0, \nabla_0)^{-1} \otimes (\mathcal{O}_Y(w D), \rmd(w D))^{-1}$. From the definition of the tensor product, one checks directly that $\Delta V=\mathcal{O}_Y$, and $\Delta\nabla\mid_{Y\setminus D_Y}=\mathrm{d}\mid_{Y\setminus D_Y}$. Thus, $(\Delta V, \Delta \nabla)=(\mathcal{O}_Y, \mathrm{d})$ and (1) holds true. (2) follows directly from (1)
\end{proof}

\subsubsection{The pullback of parabolic $\lambda$-flat bundles}
Now, we consider the pullback for  parabolic $\lambda$-flat bundles. Let $(V, \nabla)$ be a parabolic $\lambda$-flat bundle on $(Y, D_Y)/S$. For any $\gamma \in \mathbb{Q}^n$, $f^*(\gamma D)$ is also a rational divisor over $Y'$ supported on $D'$. We simply set
\[f^*\Big(\mathcal{O}_Y(\gamma D), \lambda\rmd_Y(\gamma D)\Big) := \Big(\mathcal{O}_{Y'}\big(f^*(\gamma D)\big), \lambda\rmd_{Y'}\big(f^*(\gamma D)\big)\Big).\]
Similarly, for each $\gamma \in \mathbb{Q}^n$, we set
\[f^*_{\gamma}(V, \nabla)
\coloneqq f^*\big(V(-\gamma D)_0, \nabla(-\gamma D)_0\big) \otimes f^*\big(\mathcal{O}_Y(\gamma D),\lambda \rmd_Y(\gamma D)\big).\]
Denote $U_Y=Y-D_Y$ and $U_{Y'}=Y'-D_{Y'}$. If we restrict $f^*_{\gamma}(V, \nabla)$ to the open subset $U_{Y'}$, then by the construction, we get
\[\Big(f^*_{\gamma}(V, \nabla)\Big)\mid_{U_{Y'}}=f^*\Big((V, \nabla)\mid_{U_Y}\Big).\]
By \autoref{thm_NatExtLogConn}, the connections on $f^*_{\gamma}(V, \nabla)$ for different choices of $\gamma$ coincide over the maximal common subsheaf.
\begin{definition} \label{def_pullbackdR}
Let $(V, \nabla)$ be a parabolic $\lambda$-flat bundle over $(Y, D_Y)/S$. We define the \emph{pullback of $(V, \nabla)$ along $f$} as
\[f^*(V, \nabla)\coloneqq \bigcup_{\gamma\in\mathbb{Q}^n} f^*_{\gamma}(V, \nabla).\]
\end{definition}

Notation as in \autoref{setup_PullBackPar}, given a  parabolic $\lambda$-flat bundle $(V, \nabla)$ over $(Y, D_Y)/S$ whose weights are contained in $\frac{1}{N}\mathbb{Z}$. As a sub parabolic vector bundle of $f^*\left(V_0(D_Y)\right)$, $f^*_{par}(V)$ is equipped with a connection $f^*_{par}\nabla$, which is a restriction of $f^* \nabla_0(D_Y)$.

\begin{proposition}\label{thm_ComparePullbackDeRhamBundle}
For any  parabolic $\lambda$-flat bundle $(V, \nabla)$ over $(Y, D_Y)/S$ whose weights are contained in $\frac{1}{N}\mathbb{Z}$, $f_{par}^*(V, \nabla)\cong f^*(V, \nabla)$.
\end{proposition}
\begin{proof}
By \autoref{Thm_PullBack}, $f^*V$ and $f_{par}^*V$ are the same sub vector bundle of $f^*V_0\otimes \mathcal{O}(D_Y)$. Both of $f^*_{par}\nabla$ and $f^*\nabla$ are restrictions of $f^* \nabla_0(D_Y)$.
\end{proof}

\subsection{Filtered parabolic de Rham bundles and graded parabolic Higgs bundles}
\subsubsection{Filtered parabolic de Rham vector bundle}
In this paper, a filtration $\Fil^\bullet$ on a pair $(V,\nabla)$
consisting of a parabolic sheaf $V$ and a parabolic connection $\nabla$ on $V$ over $(Y,D_Y)/S$ will be called a \emph{Hodge filtration of level in $[a,b]$} if the following conditions hold:
\begin{itemize}
\item[-] $\Fil^i V$'s are parabolic locally split sub-sheaves of $V$, with
\[V=\Fil^aV\supset \Fil^{a+1}V \supset\cdots \supset \Fil^bV\supset \Fil^{b+1}V=0.\]
\item[-] $\Fil$ satisfies Griffiths transversality with respect to the parabolic connection $\nabla$.
\end{itemize}
In this case, the triple $(V,\nabla,\Fil)$ is called a \emph{filtered parabolic de Rham bundle} if $(V,\nabla)$ is a parabolic de Rham bundle.

\subsubsection{Graded parabolic vector bundles}
Let $\{\Gr^\ell E\}_{\ell\in \mathbb{Z}}$ be parabolic subbundles of a parabolic vector bundle $E$ on $Y$. The pair $(E, \Gr)$ is called \emph{graded parabolic vector bundle} over $Y$ if the natural map $ \oplus_{\ell\in \mathbb{Z}} \Gr^\ell E\cong E$ is an isomorphism.

\subsubsection{Graded parabolic Higgs bundles}
A \emph{graded parabolic Higgs bundle} over $(Y, D_Y)/S$ is a parabolic Higgs bundle $(E, \theta)$ together with a grading structure $\Gr$ on $E$ satisfying
\[\theta(\Gr^\ell E) \subset \Gr^{\ell -1}E\otimes_{\mathcal{O}_Y} \Omega^1_{Y/S}(\log D_Y).\]
The following is the main example we will be concerned with.
\begin{example}
Let $(V, \nabla, \Fil)$ be a filtered parabolic de Rham bundle over $(Y,D_Y)/S$. Let $$E =\bigoplus_{\ell\in \mathbb{Z}} \Fil^\ell V/\Fil^{\ell+1}V,\quad \Gr^\ell E = \Fil^\ell V/\Fil^{\ell+1}V.$$
By Griffith's transversality, for each $\ell\in \mathbb{Z}$, the connection induces an $\mathcal{O}_Y$-linear map $$\theta \colon \Gr^\ell E \rightarrow \Gr^{\ell-1} E \otimes_{\mathcal{O}_Y} \Omega^1_{Y/S}(\log D_Y).$$
Then $(E, \theta, \Gr)$ is a graded parabolic Higgs bundle.
\end{example}

\section{\bf Parabolic Fontaine-Faltings Modules and Parabolic Crystalline Representations} \label{sec_Para_FF}

In this section, we introduce the definitions of parabolic Fontaine-Faltings modules and parabolic crystalline representations.

\begin{setup}\label{setup_PFF and PHDF}
    Let $f$ be a positive integer and let $k$ is a perfect field of characteristic $p>2$ containing $\mathbb F_{p^f}$. Let $W$ be the ring of Witt vectors over $k$ and $W_n=W/p^nW$. Let $S=\Spec W$ and $S_n=\Spec W_n$. Let $Y$ be a smooth separated  scheme over $W$ with a geometrically connected generic fiber. Let $D_Y\subseteq Y$ be a relative normal crossing divisor $D_Y$ with complement $U_Y$. Denote by
\[(Y_n, D_{Y_n})\coloneqq (Y, D_Y)\times_SS_n.\]
Let $\mathcal {(Y,D_Y,U_Y)}$ be the $p$-adic formal completion of ${(Y,D_Y,U_Y)}$ with rigid generic fiber $(\mathcal Y_K,\mathcal{D}_{\mathcal Y_K},\mathcal{U}_{\mathcal Y_K} )$. Denote by
\[\mathcal Y_K^\circ\coloneqq\mathcal Y_K-\mathcal{D}_{\mathcal Y_K}.\]

\end{setup}

\subsection{Introduction to Fontaine-Faltings modules}
In this subsection, we review the notion of Fontaine-Faltings module.
We follow closely (both in notation and in exposition) the foundational work of Faltings \cite{Fal89}.
See also \cite[Section 1]{SYZ22} and \cite[Section 2]{LSZ19}.

\subsubsection{Fontaine-Faltings modules over a small affine base.}
\label{sect:FF/small}

We first recall the notion of a Fontaine-Faltings module over a small affine scheme.
Assume $Y$ is a \emph{small} affine scheme,
i.e., if $Y=\mathrm{Spec}(R)$,
then there exists an \'etale map
\[W[T_1^{\pm1},T_2^{\pm1},\cdots, T_{d}^{\pm1}]\rightarrow R,\]
over $W$ (see \cite[p. 27]{Fal89}).
In general, a smooth affine scheme over $W$ is not always small but it can be covered by a system of small affine open subsets.
By the existence of the \'etale chart there exists some $\Phi:\widehat{R}\rightarrow\widehat{R}$ which lifts the absolute Frobenius on $R/pR$, where $\widehat{R}$ is the $p$-adic completion of $R$.

A \emph{Fontaine-Faltings module} over the $p$-adic formal completion $\mathcal Y=\mathrm{Spf}(\widehat{R})$ of $Y$ with Hodge-Tate weights in $[a,b]$ is a quadruple $(V,\nabla,\Fil,\varphi)$, where
\begin{itemize}
\item[-] $(V,\nabla)$ is a de Rham $\widehat{R}$-module;
\item[-] $\Fil$ is a Hodge filtration on $(V,\nabla)$ of level in $[a,b]$;
\item[-] $\widetilde{V}$ is the quotient $\bigoplus\limits_{i=a}^b\Fil^i/\sim$. Here, $px\sim y$ for $x\in\Fil^iV$ where $y$ is the image of $x$ under the natural inclusion $\Fil^iV\hookrightarrow\Fil^{i-1}V$;
\item[-] $\varphi$ is an $\widehat{R}$-linear isomorphism \[\varphi:\widetilde{V}\otimes_{\Phi}\widehat{R} \longrightarrow V,\]
\item[-] The relative Frobenius $\varphi$ is horizontal with respect to the connections.
\end{itemize}
(The fact that $\varphi$ is an isomorphism is sometimes known as \emph{strong $p$-divisibility.}) A morphism between Fontaine-Faltings modules is a morphism between the underlying de Rham modules which is strict for the filtrations and commutes with the $\varphi$-structures. Denote by $\mathcal {MF}_{[a,b]}^{\nabla,\Phi}(\mathcal Y/W)$ the category of all Fontaine-Faltings modules over $\mathcal Y$ of Hodge-Tate weights in $[a,b]$. The $p$-primary torsion version of this definition was first written down in \cite[p. 30-31]{Fal89}; here we follow \cite[Section 3]{Fal99}, see also \cite[Section 2]{SYZ22} and \cite[Section 2]{LSZ19}.

\subsubsection{The gluing functor.} In the following, we recall the gluing functor of Faltings. In other words, up to a canonical equivalence of categories, if $b-a\leq p-2$, the category $\mathcal {MF}_{[a,b]}^{\nabla,\Phi}(\mathcal Y/W)$ does not depend on the choice of $\Phi$. More explicitly, the functor yielding an equivalence is given as follows. Let $\Psi$ be another lifting of the absolute Frobenius. For any filtered de Rham module $(V,\nabla,\Fil)$, Faltings~\cite[Theorem~2.3]{Fal89} shows that there is a canonical isomorphism by the Taylor formula
\[\alpha_{\Phi,\Psi}: \widetilde{V}\otimes_\Phi\widehat{R} \simeq \widetilde{V}\otimes_\Psi \widehat{R},\]
which is \emph{compatible} with respect to the connection, satisfies the cocycle conditions and induces an equivalence of categories
\begin{equation}
\xymatrix@R=0mm{ \mathcal {MF}_{[a,b]}^{\nabla,\Psi}(\mathcal Y/W)\ar[r] & \mathcal {MF}_{[a,b]}^{\nabla,\Phi}(\mathcal Y/W).\\
(V,\nabla,\Fil,\varphi)\ar@{|->}[r] & (V,\nabla,\Fil,\varphi\circ\alpha_{\Phi,\Psi})\\}
\end{equation}

\subsubsection{Fontaine-Faltings modules over a global base.}
\label{sect:FF/gl}

In this section, we do not assume $Y$ is small.
Let $\{\mathcal U_i\}_{i\in I}$ be a connected small affine open covering of $\mathcal Y$
with $\Phi_i$ a lift of the absolute Frobenius on $\mathcal O_\mathcal Y(\mathcal U_i)\otimes_W k$ for any $i\in I$.
Recall that the category $\mathcal {MF}_{[a,b]}^{\nabla}(\mathcal Y/W)$ is constructed by gluing the categories $\mathcal {MF}_{[a,b]}^{\nabla,\Phi_i}(\mathcal U_i/W)$.
Actually $\mathcal {MF}_{[a,b]}^{\nabla}(\mathcal Y/W)$ can be described more precisely as follows.
A Fontaine-Faltings module over $\mathcal Y$ of Hodge-Tate weights in $[a,b]$ is a tuple $(V,\nabla,\Fil,\{\varphi_i\}_{i\in I})$,
i.e., a filtered de Rham sheaf $(V,\nabla,\Fil)$ over $\mathcal Y$ together with
$\varphi_i: \widetilde{V}(\mathcal U_i)\otimes_{\Phi_i} \widehat{\mathcal O_\mathcal Y(\mathcal U_i)}\rightarrow V(\mathcal U_i)$
such that

\begin{itemize}
\item[-] $M_i:=(V(\mathcal U_i),\nabla,\Fil,\varphi_i)\in \mathcal {MF}_{[a,b]}^{\nabla,\Phi_i}(\mathcal U_i/W)$.

\item[-] For all $i,j\in I$, on the open intersection $\mathcal U_i\cap \mathcal U_j$, the Fontaine-Faltings modules $M_i\mid_{\mathcal U_{i}\cap \mathcal U_j}$ and $M_j\mid_{\mathcal U_{i}\cap \mathcal U_j}$ are associated to each other under the above equivalence of categories with respect to the two Frobenius liftings $\Phi_i$ and $\Phi_j$ on $\mathcal U_i\cap \mathcal U_j$.
\end{itemize}
Denote by $\mathcal {MF}_{[a,b]}^{\nabla}(\mathcal Y/W)$ the category of all Fontaine-Faltings modules over $\mathcal Y$ of Hodge-Tate weights in $[a,b]$. For any Fontaine-Faltings module $M=(V,\nabla,\Fil,\{\varphi_i\}_{i\in I})$ on $\mathcal Y$, $V$ is locally isomorphic to a direct sum of sheaves of the form $\mathcal O_{\mathcal Y}/p^e$ by \cite[Theorem 2.1]{Fal89}. $M$ is called a \emph{strict $p^n$-torsion Fontaine-Faltings module} if $V$ is a locally free $\mathcal O_{Y_n}$-module. A \emph{$W_n(\mathbb F_{p^f})$-endomorphism structure on a Fontaine-Faltings module $M$}  is an embedding of $\mathbb Z/p^n \mathbb Z$-algebra
\[W_n(\mathbb F_{p^f})\hookrightarrow \operatorname{End}(M).\]

\subsection{Parabolic Fontaine-Faltings modules}\label{sec_PFFandPHDF}
To define parabolic Fontaine-Faltings modules, we need the parabolic versions of Faltings' tilde functor $\widetilde{(\cdot)}$ and Frobenius pullback functor $\mathcal{F}_n$. We first introduce the following parabolic versions of the categories.
\begin{definition}
    \begin{itemize}
        \item Denote by \emph{$\MCF^{par}_{p-2}(Y_n)$} the category of triples $(V,\nabla,\Fil)$, each consisting of a parabolic vector bundle $V$ on $(Y_n,D_{Y_n})/S_n$, a parabolic connection $\nabla$ on $V$ (not necessarily integrable) and Hodge filtration $\Fil$  of level in $[0,p-2]$.
        \item Denote by \emph{$\MIC^{par}(Y_n)$} the category of all parabolic de Rham bundles over $(Y_n, D_{Y_n})/S_n$.
        \item Denote by \emph{$\widetilde{\MIC}^{par}(Y_n)$} the category of all parabolic $p$-flat bundles (i.e. $\lambda$-flat bundles with $\lambda=p$) over $(Y_n, D_{Y_n})/S_n$.
        \item Denote by \emph{$\widetilde{\MC}^{par}(Y_n)$} the category  whose objects satisfy the same conditions as those of $\widetilde{\MIC}^{par}(Y_n)$, except that the integrability condition is not imposed.
    \end{itemize}
\end{definition}

\subsubsection{Faltings' tilde functor $\widetilde{(\cdot)}$}
For an object $(V, \nabla, \Fil)$ in $\MCF^{par}_{p-2}(Y_n)$ which has  trivial parabolic structure, $\widetilde{V}$ is denoted as the quotient $\bigoplus\limits_{i=0}^{p-2}\Fil^i/\sim$ with $px\sim y$ for any $x\in \Fil^iV$ and $y$ the image of $x$ under the natural inclusion $\Fil^iV\hookrightarrow \Fil^{i-1}V$. Let $s$ be a section of $\Fil_i$ and denote by $\tilde{s}$(resp. $s'$) its image under $\Fil^i\to\widetilde{V}$(resp. $\Fil^i\to\Fil^{i-1}$).
$\nabla$ induces a $\mathcal O_{Y_n}$-linear map $\widetilde{\nabla}:\widetilde{V}\to \widetilde{V}\otimes \Omega_{Y_n/S_n}(\log D_{Y_n})$ by setting $$\widetilde\nabla(\tilde{s})\coloneqq \widetilde{\nabla(s)}.$$
For any local section $f$ of $\mathcal O_{Y_n}$,
\[\widetilde\nabla(f\tilde{s})=\widetilde{\nabla(fs)}=f\widetilde{\nabla(s)}+\widetilde{s'}\otimes \rmd f=f\widetilde{\nabla}(\tilde{s})+p\tilde{s}\otimes \rmd f,\] thus $\widetilde{\nabla}$ is a $p$-connection. Then
\[\widetilde{(V,\nabla,\Fil)}\coloneqq (\widetilde V,\widetilde \nabla).\]
The functor  $\widetilde{(\cdot)}$ preserves integrability condition.
More generally, for any object $(V, \nabla, \Fil)$ in $\MCF^{par}_{p-2}(Y_n)$ and $\alpha\in \mathbb Q^n$, we denote
\[\widetilde{(V, \nabla, \Fil)}_{\alpha}:= \widetilde{(V_{\alpha}, \nabla_{\alpha}, \Fil_{\alpha})}.\]
By the following lemma, $\{\widetilde{(V, \nabla, \Fil)}_{\alpha}\}_{\alpha \in \mathbb Q^n}$ forms a parabolic $p$-connection. Thus we can define the Faltings' tilde functor $\widetilde{(\cdot)}$ for parabolic bundle by setting
\[\widetilde{(V, \nabla, \Fil)}\coloneqq \{\widetilde{(V, \nabla, \Fil)}_{\alpha}\}_{\alpha\in \mathbb Q^n}.\]

\begin{lemma}\label{Prop_FaltingTilde}
    The collection of pairs $\{\widetilde{(V, \nabla, \Fil)}_{\alpha}\}_{\alpha\in \mathbb Q^n}$, each comprising a vector bundle and a $p$-connection, forms an object in $\widetilde{\MC}^{par}(Y_n)$.
\end{lemma}
\begin{proof}
    By the functoriality of $\widetilde{(\cdot)}$, it suffices to verify that $\{\widetilde V_{\alpha}\}_{\alpha\in \mathbb Q^n}$ is a parabolic bundle.
    The problem is local on $Y_n$, we can assume that $\Fil^i$ ($i=0,1,\ldots, p-2$) is isomorphic to a direct sum of parabolic line bundles and splitting. Moreover, as the functor $\widetilde{(\cdot)}$ preserves direct sum of vector bundles, we can further assume that $V\cong \mathcal O_{Y_n}(\gamma D_{Y_n})$ is a parabolic line bundle with $\gamma \in ([0,1)\cap \mathbb Q)^n$,   $\Fil^k=V$  and $\Fil^{k+1}=0$. Then for any $\alpha \in \mathbb Q^n$,
    \[\widetilde V_{\alpha}= \mathcal O_{Y_n}(\lfloor \alpha+\gamma\rfloor D_{Y_n})^{\sim}\cong \mathcal O_{Y_n}(\lfloor \alpha+\gamma\rfloor D_{Y_n}).\]
    Thus $\widetilde V$ is parabolic bundle sharing the same parabolic weights as $V$.
\end{proof}
\begin{remark}
   By \autoref{Prop_FaltingTilde},  $\widetilde{(\cdot)}$ is restricted to be a functor from the category $\MIC^{par}_{p-2}(Y_n)$ to the category
   $\widetilde{\MIC}^{par}(Y_n)$.
\end{remark}

\subsubsection{The Frobenius pullback functor $\mathcal{F}_n$}
Let $(\widetilde{V}, \widetilde{\nabla})$ be an object in $\widetilde{\MIC}^{par}(Y_n)$, i.e., a parabolic $p$-flat bundle.  Since $(\widetilde{V}_\alpha, \widetilde{\nabla}_\alpha)$ is $p$-flat bundle with trivial parabolic structure for any $\alpha\in \mathbb Q^n$, by usual Frobenius pullback in \cite[Proposition 4.15]{LSZ19}, one obtains a logarithmic de Rham bundle
\[\mathcal F_n(\widetilde{V}_\alpha, \widetilde{\nabla}_\alpha)\]
on $(Y_{n},D_{Y_n})$.  Since the parabolic structure are supported on $D_{Y_n}$. the restriction of  $(\widetilde{V}_\alpha, \widetilde{\nabla}_\alpha)$ to $U_{Y_n}=Y_{n} - D_{Y_n}$ is a usual vector bundle and independent of the choice of $\alpha$, we simply denote it by $(\widetilde{V}, \widetilde{\nabla})\mid_{U_{Y_n}}$.
For any $\alpha\in \mathbb Q^n$,
$$\mathcal F_n(\widetilde{V}_\alpha, \widetilde{\nabla}_\alpha)\mid_{U_{Y_n}}
= \mathcal F_n(\widetilde{V}\mid_{U_{Y_n}}, \widetilde{\nabla}\mid_{U_{Y_n}}).$$
Thus $\mathcal F_n(\widetilde{V}_\alpha, \widetilde{\nabla}_\alpha)$ can be viewed as a sub de Rham sheaf of
\[j_{Y_n *}\left(\mathcal F_n (\widetilde{V}\mid_{U_{Y_n}}, \widetilde{\nabla}\mid_{U_{Y_n}})\right),\]
where $j_{Y_n}\colon U_{Y_n}\rightarrow Y_n$ is the natural open embedding.

Let $\mathcal{U}=\{U_{n+1,\mu}\}_{\mu\in \Lambda}$ be an open covering of $Y_{n+1}$ such that for any $U_{n+1, \mu}$, there is a morphism ${\Phi}_{n+1, \mu}:U_{n+1, \mu}\to U_{n+1, \mu}$ which lifts the absolute Frobenius on ${U_{n,\mu}}\coloneqq U_{n+1, \mu}\otimes \mathbb{Z}/p\mathbb{Z}$ and preserves the divisor $D_{Y_{n+1}}$ in the sense that $$\Phi_{n+1,\mu}^*(D_{Y_{n+1}}\cap U_{n+1,\mu}) = p(D_{Y_{n+1}}\cap U_{n+1,\mu}).$$
Set
\[U_{n,\mu}\coloneqq U_{n+1, \mu}\otimes \mathbb{Z}/p^n\mathbb{Z}, \quad \Phi_{n,\mu}\coloneqq \Phi_{n+1,\mu}\otimes \mathbb{Z}/p^n\mathbb{Z}, \quad
V\mid_{U_{n,\mu}}\coloneqq \Phi_{n,\mu}^*(\widetilde{V}\mid_{U_{n,\mu}}).\]
By \autoref{prop:PullbackParabolicBundles},
\[V\mid_{U_{n,\mu}} = \bigcup_{\gamma\in\mathbb Q^n} \Phi_{n,\mu}^*\left((\widetilde{V}(-\gamma D_{Y_n})_0)\mid_{U_{n,\mu}}\right)\otimes \Phi_{n,\mu}^*(\mathcal O_Y(\gamma D_{Y_n})\mid_{U_{n,\mu}}) \]
 Since the ${\Phi}_{n, \mu}$ preserves the divisor $D_{Y_n}$, one has
 \[\Phi_{n,\mu}^*(\mathcal O_Y(\gamma D_{Y_n})\mid_{U_{n,\mu}}) = \mathcal O_Y(p\gamma D_{Y_n})\mid_{U_{n,\mu}}.\]
Denote by
\[(V_{\gamma,\mu},\nabla_{\gamma,\mu})\coloneqq \mathcal F_{n} \left(\widetilde{V}(-\gamma D_{Y_n})_0 \mid_{U_{n,\mu}}, \widetilde{\nabla}(-\gamma D_{Y_n})_0 \mid_{U_{n,\mu}}\right).\]
Then we can defined the Frobenius pullback functor for $(\widetilde{V}, \widetilde{\nabla})\mid_{U_{n,\mu}}$ by setting
\[\mathcal F_n\left((\widetilde{V}, \widetilde{\nabla})\mid_{U_{n,\mu}}\right)\coloneqq \bigcup_{\gamma\in\mathbb Q^n} (V_{\gamma,\mu},\nabla_{\gamma,\mu}) \otimes \big( \mathcal{O}_Y(p\gamma D_{Y_n}), \rmd(p\gamma D_{Y_n})\big)\mid_{U_{n,\mu}}. \]
Since $(\widetilde V,\widetilde \nabla)$ can be regarded as a sub parabolic $p$-flat sheaf of  $(\widetilde V',\widetilde \nabla')\coloneqq(\widetilde V_0(D_{Y_n}),\widetilde \nabla_0(D_{Y_n}))$, and there exists a family of isomorphisms
\[\left\{G_{\mu\tau}:\mathcal F_n\left(\big(\widetilde V',\widetilde \nabla'\big)\big|_{U_{n,\mu}}\right)\big|_{U_{n,\mu}\cap U_{n,\tau}} \cong \mathcal F_n\left(\big(\widetilde V',\widetilde \nabla'\big)\big|_{U_{n,\tau}}\right)\big|_{U_{n,\mu}\cap U_{n,\tau}}\right\}_{\mu,\tau\in \Lambda}\]
satisfying the cocycle condition and gluing $\left\{\mathcal F_n\left(\big(\widetilde V',\widetilde \nabla'\big)\big|_{U_{n,\mu}}\right)\right\}_{\mu\in \Lambda}$ to be a global object $\mathcal F_n(\widetilde V',\widetilde \nabla')$ by \cite[Proposition 4.15]{LSZ19}, $\{G_{\mu\tau}\}_{\mu,\tau\in \Lambda}$ also glues $\left\{\mathcal F_n\left(\big(\widetilde V,\widetilde \nabla\big)\big|_{U_{n,\mu}}\right)\right\}_{\mu\in \Lambda}$ to be a  parabolic de Rham bundle which we denote by $\mathcal F_n(\widetilde V,\widetilde \nabla)$. It is a sub parabolic de Rham sheaf of $\mathcal F_n(\widetilde V',\widetilde \nabla')$. Thus we have defined the parabolic Frobenius pullback functor from the category $\widetilde{\MIC}^{par}(Y_n)$ to the category
$\MIC^{par}_{p-2}(Y_n)$.

\subsubsection{Parabolic Fontaine-Faltings modules with endomorphism structure}
Recall \cite[Lemma 1.1]{SYZ22} or \cite[Lemma 5.6]{LSZ19}. We extend the definition to the parabolic version.
\begin{definition}
A \emph{strict $p^n$-torsion parabolic Fontaine-Faltings module over $(Y_n, D_{Y_n})/S_n$} is a tuple $(V, \nabla, \Fil, \varphi)$, where
\begin{itemize}
\item $(V, \nabla, \Fil)$ is a filtered parabolic de Rham bundle over $(Y_n, D_{Y_n})/S_n$ of level in $[0, p-2]$;
\item $\varphi\colon \mathcal{F}_n\widetilde{(V,\nabla,\Fil)}\rightarrow (V, \nabla)$ is an isomorphism of parabolic de Rham bundles.
\end{itemize}
We denote by \emph{$\MF_{[0, p-2]}^\nabla((Y_n, D_{Y_n})/S_n)$} the category of all strict $p^n$-torsion parabolic Fontaine-Faltings modules over $(Y_n, D_{Y_n})/S_n$. By taking inverse limit, one obtains the category of \emph{torsion-free parabolic Fontaine-Faltings modules}  \emph{$\MF_{[0, p-2]}^\nabla((\mathcal Y, \mathcal{D_{Y}})/S)$}, whose objects are tuples of the form $(V, \nabla, \Fil, \{\varphi_n\}_{n\in\mathbb Z_{\geq 0}})$ such that
\begin{itemize}
\item $(V, \nabla, \Fil)$ is a filtered parabolic de Rham bundle over $(\mathcal Y, \mathcal{D_{Y}})/S$ of level in $[0, p-2]$;
\item $\left((V, \nabla, \Fil)\otimes \mathbb Z/p^n \mathbb Z,\varphi_n\right)$ is an object in $\MF_{[0, p-2]}^\nabla((Y_n, D_{Y_n})/S_n)$ for any $n\in \mathbb Z_{\geq 0}$;
\item $\varphi_{n-1}=\varphi_{n}\otimes \mathbb Z /p^{n-1}\mathbb Z$.
\end{itemize}
\end{definition}

\begin{definition}
Let $M\in \MF_{[0, p-2]}^\nabla((Y_n, D_{Y_n})/S_n)$. A \emph{$W_n(\mathbb F_{p^f})$-endomorphism structure} is an
embedding of $\mathbb Z/p^n\mathbb Z$-algebra
\[\iota\colon W_n(\mathbb F_{p^f})\hookrightarrow \End(M).\]
Denote by \emph{$\MF_{[0, p-2],f}^\nabla((Y_n, D_{Y_n})/S_n)$} the category of all strict $p^n$-torsion parabolic Fontaine-Faltings modules with $W_n(\mathbb F_{p^f})$-endomorphism structures over $(Y_n, D_{Y_n})/S_n$.

Similarly, when $M$ is a torsion-free parabolic Fontaine-Faltings module over $(\mathcal Y,\mathcal{D_Y})/S$, A \emph{$W(\mathbb F_{p^f})$-endomorphism structure} of $M$ is an embedding of $\mathbb Z_p$-algebra
\[\iota\colon W(\mathbb F_{p^f})\hookrightarrow \End(M).\]
Denote by \emph{$\MF_{[0, p-2],f}^\nabla((\mathcal Y, \mathcal{D_{Y}})/S)$} the the category of all torsion-free parabolic Fontaine-Faltings modules with $W(\mathbb F_{p^f})$-endomorphism structures over $(\mathcal Y, \mathcal{D_{Y}})/S$.
\end{definition}

Let $(V, \nabla, \Fil, \varphi)$ be a strict $p^n$-torsion parabolic Fontaine-Faltings module with level in $[0, a]$ and $(V, \nabla, \Fil, \varphi)'$ be another strict $p^n$-torsion parabolic Fontaine-Faltings module with level in $[0, b]$ such that $a+b\leq p-2$. Then on the tensor parabolic de Rham bundle
\[(V, \nabla)\otimes (V, \nabla)'.\]
Set
\[\Fil_{tot}^n(V\otimes V'):= \sum_{i+j=n}\Fil^iV\otimes \Fil^jV\quad \text{for all $n\in\mathbb{Z}$},\]
it forms a Hodge filtration on the parabolic de Rham bundle $(V, \nabla)\otimes (V, \nabla)'$. Denote
\[(V, \nabla, \Fil)\otimes (V, \nabla, \Fil)':=(V\otimes V', \nabla\otimes\id +\id\otimes\nabla', \Fil_{tot}).\]
By the definition of the parabolic version of the inverse Cartier functor, it naturally preserves tensor products, so we can set a Frobenius structure on $(V, \nabla, \Fil)\otimes (V, \nabla, \Fil)'$ given by $\varphi\otimes\varphi'$, and then obtain the tensor product
\[(V, \nabla, \Fil, \varphi) \otimes (V, \nabla, \Fil, \varphi)'.\]
For the tensor product of Fontaine-Faltings modules with endomorphism structure, we set the underlying parabolic Fontaine-Faltings module of
\[(V, \nabla, \Fil, \varphi, \iota) \otimes (V, \nabla, \Fil, \varphi, \iota)'\]
to be
\[\Big((V, \nabla, \Fil, \varphi) \otimes (V, \nabla, \Fil, \varphi)'\Big)^{\iota=\iota'}.\]
on this, the action of $\iota$ and $\iota'$ coincides, and we use this common action to define the $\mathbb{Z}_{p^f}$-endomorphism structure.

We also define the symmetric product, wedge product, and determinant for parabolic Fontaine-Faltings modules in the usual way.

\subsection{Parabolic $\mathbb D^{par}$-functor and parabolic crystalline representations}
As a $p$-adic analogous of Hilbert-Riemann correspondence, Faltings introduced a fully faithful functor $\mathbb D$ from the category of Fontaine-Faltings modules over $\mathcal Y$ to the category of  continuous $\mathbb Z_p$-representations of $\pi_1^{\et}(\mathcal Y_K)$ (\cite[Theorem 2.6*]{Fal89}).
As a generalization of Faltings' work,
we constructed the logarithmic $\mathbb D^{\log}$-functor from the category of logarithmic Fontaine-Faltings modules over $\mathcal {(Y,D_Y)}$ to the category of continuous $\mathbb Z_p$-representations of $\pi_1^{\et}(\mathcal Y_K^\circ)$ in \cite[Theorem 2.9]{LYZ25}.

For any $p^n$-torsion logarithmic Fontaine-Faltings module $M$ over $\mathcal {(Y,D_Y)}$, the construction of the $\mathbb Z_p$-representations $\mathbb D^{\log}(M)$ of $\pi_1^{\et}(\mathcal Y_K^\circ)$ is  divided into the following two steps:

\begin{enumerate}[Step 1:]
    \item
    First assuming that that $Y$ is small and affine. Choose a finite flat cover $\pi: \mathcal Y'\to \mathcal Y$  such that $\pi$ is a \'etale outside $\mathcal Y'-\pi^{-1}(\mathcal{D_Y} )$ and the coefficients of the pullback divisor $\pi^*\mathcal D_{\mathcal Y}$ is divisible by $p^n$. Then $\pi^*M$ has no poles (\cite[page 15]{LYZ25}).
    Applying the Faltings' $\mathbb D$-functor to the usual Fontaine-Faltings module $\pi^*M$ on $\mathcal Y'$, one gets a $\mathbb Z_p$-module $\mathbb D(\pi^*M)$ which is equipped with compatible actions of
    $\pi_1^{\et}(\mathcal U_{\mathcal Y_K})$ and  $\pi_1^{\et}(\mathcal Y'_K)$. By \cite[Lemma 2.4]{LYZ25}, the two compatible actions induced an action of $\pi_1^{\et}({\mathcal Y_K^\circ})$.
    \item
    Generally, choose a small affine covering of $Y$.  On each open subset $U$ , one gets a $\mathbb Z_p$-representation of $\pi_1^{\et}(\mathcal U_K)$, where $\mathcal U_K$ is the rigid generic fiber of the $p$-adic formal completion $\mathcal U$ of $U$. By \cite[Proposition 2.7]{LYZ25}, these local representations are compatible and thus can be glued to form  a global representation of  $\pi_1^{\et}(\mathcal Y_K^\circ)$.
\end{enumerate}

The same method can be applied to a strict $p^n$-torsion parabolic Fontaine-Faltings modules $M$ whose all parabolic weights are contained in $\frac{1}{N}\mathbb Z$.
The only difference is to replace the cover $\pi$ by a cover $\pi':\mathcal Y''\to\mathcal Y$ such that $\pi$ is a \'etale outside $\mathcal Y'-\pi'^{-1}(\mathcal{D_Y} )$ and the coefficients of the pullback divisor $\pi'^*\mathcal D_{\mathcal Y}$ is divisible by $N p^n$. In this case, $\pi'^*M$ is Fontaine-Faltings module with no poles and trivial parabolic structure. The remaining part of the proof follows that in \cite{LYZ25}, and eventually we can prove the following theorem:

\begin{theorem}\label{thm_ParabolicDFunctor}
We keep the same notation as in \autoref{setup_PFF and PHDF}.
\begin{enumerate}
\item
There is a functor $\mathbb D^{par}$ from the category strict $p^n$-torsion parabolic Fontaine-Faltings modules over $(Y_n, D_{Y_n})/S_n$ (reps. torsion-free parabolic Fontaine-Faltings modules over $(\mathcal Y,\mathcal{D_Y})/S$)
to the category of continuous $\mathbb Z/p^n \mathbb Z$-representations (reps. $\mathbb Z_p$-representations) of $\pi_1^{\et}(\mathcal Y_K^\circ)$. This functor is compatible with usual Faltings' $\mathbb D$-functor in the following sense: for any strict $p^n$-torsion Fontaine-Faltings modules $M$ over $(Y_n,D_{Y_n})/S_n$ (reps. torsion-free parabolic Fontaine-Faltings modules $M$ over $(\mathcal Y,\mathcal{D_Y})/S$), there is a canonical isomorphism of continuous $\mathbb Z/p^n \mathbb Z$-representation (reps. $\mathbb Z_p$-representations) of $\pi_1^{\et}(\mathcal U_{\mathcal Y_K})$:
\[\mathbb D^{par}(M)\mid_{\mathcal U_{\mathcal Y_K}}\cong \mathbb D(M\mid_{ Y_n-D_{Y_n}}) \quad \left(\text{reps. }\mathbb D^{par}(M)\mid_{\mathcal U_{\mathcal Y_K}}\cong \mathbb D(M\mid_{ \mathcal Y-\mathcal{D_Y}})\right).\]
The representations in the essential image of $\mathbb D^{par}$ are called \emph{parabolic crystalline representations}.

\item
If $Y$ is proper over $W$, then the functor in (1) can be algebraized  into a functor from the category strict $p^n$-torsion parabolic Fontaine-Faltings modules over $(Y_n, D_{Y_n})/S_n$ (reps. torsion-free parabolic Fontaine-Faltings modules over $(\mathcal Y,\mathcal{D_Y})/S$)
to the category of continuous $\mathbb Z/p^n \mathbb Z$-representations (reps. $\mathbb Z_p$-representations) of $\pi_1^{\et}(Y_K^\circ)$, where $Y_K^\circ$ is the generic fiber of $Y-D_Y$.
\end{enumerate}
\end{theorem}

\begin{remark}
Faltings constructed the $\mathbb D$-functor for general Fontaine-Faltings modules. Here, due to technical limitations, we only construct the $\mathbb D^{par}$-functor for parabolic strict $p^n$-torsion and torsion-free Fontaine-Faltings modules, which is sufficient for applications.
\end{remark}

\section{\bf Parabolic Higgs-de Rham Flows}
\label{sec_Para_Higgs_de_Rham}

In this section, we introduce the definitions of parabolic Higgs-de Rham flows over $W_n$ or $W$. One see \cite{KrSh20} for a positive characteristic version of construction of parabolic Higgs-de Rham flows on curves. We maintain the same setup as in \autoref{sec_Para_FF}.

\subsection{Introduction to Higgs-de Rham flows}

Lan Sheng and Zuo have introduced the notion of (periodic) Higgs-de Rham flow over a smooth scheme $Y$ over $W(k)$ and have shown that the category of $f$-periodic Higgs bundles is equivalent to the category of Fontaine-Faltings modules over $Y$ endowed with a $\mathbb{Z}_{p^f}$-endomorphism structure in \cite{LSZ19}. The theory of Higgs-de Rham flows is pivotal in various research domains, including the study of the positive Bogomolov inequality \cite{Lan15} the theory of rigid connections \cite{EsGr20}, and other related areas.

Lately, Sun-Yang-Zuo introduced twisted Higgs-de Rham flow, which generalized the notion of Higgs-de Rham flow in \cite{SYZ22}. Actually, it is just the parabolic Higgs-de Rham flow with only one parabolic weight along each component of the boundary divisor.

Recall from \cite{LSZ19} that a \emph{Higgs-de Rham flow} over $Y_1$ is a sequence consisting of infinitely many alternating terms of filtered de Rham bundles and Higgs bundles:
\[\left\{
(V, \nabla, \Fil)_{-1},
(E, \theta)_{0},
(V, \nabla, \Fil)_{0},
(E, \theta)_{1},
(V, \nabla, \Fil)_{1},
\cdots\right\},\]
which are related to each other by the following diagram inductively
\begin{equation} \label{diag:HDF}
\xymatrix@W=10mm@C=3mm@R=5mm{
(V, \nabla, \Fil)_{-1} \ar[dr]^{\Gr} && (V, \nabla, \Fil)_{0}\ar[dr]^{\Gr}
&& (V, \nabla, \Fil)_{1}\ar[dr]^{\Gr}
&\\
& (E, \theta)_{0} \ar[ur]^{\mathcal{C}_1^{-1}}
&& (E, \theta)_{1} \ar[ur]^{\mathcal{C}_1^{-1}}
&& \cdots\\
}
\end{equation}
where
\begin{itemize}
\item[-] $(V, \nabla, \Fil)_{-1}$ is a filtered logarithmic de Rham bundle over $Y_1$ of level in $[0, p-2]$;
\item[-] $\mathcal C_1^{-1}$ is the inverse Cartier functor  in \cite[Theorem 2.8]{OgVo07}.
\item[-] Inductively, for $i\geq1$, $(E, \theta)_i$ is the graded logarithmic Higgs bundle $\Gr(V, \nabla, \Fil)_{i-1}$, and
\[(V, \nabla)_i:=\mathcal{C}_1^{-1} (E, \theta)_i,\]
which is endowed a Hodge filtration $\Fil_i$ of level in $[0, p-2]$.
\end{itemize}

\begin{remark}
\begin{enumerate}
\item In \cite{LSZ19}, the Higgs-de Rham flow and inverse Cartier functor are defined over any truncated level, and their definitions are considerably more complicated.
\item The essential data in a Higgs-de Rham flow are just
\[V_{-1}, \nabla_{-1}, \Fil_{-1}, \Fil_{0}, \Fil_{1}, \Fil_{2}, \cdots\]
since the other terms can be constructed from these, e.g., $E_0$, $\theta_0$, $V_1$, $\nabla_1$, $\cdots$.
\end{enumerate}
\end{remark}

A Higgs-de Rham flow over $Y_1$ is called $f$-periodic if there exists an isomorphism
\[\Phi\colon (E_{f}, \theta_{f}, V_{f-1}, \nabla_{f-1}, \Fil_{f-1}, \mathrm{id}) \cong (E_0, \theta_0, V_{-1}, \nabla_{-1}, \Fil_{-1}, \mathrm{id})\]
such that its induces isomorphisms, for all $i\geq0$,
\[(E, \theta)_{f+i}\cong (E, \theta)_{i} \quad\text{and}\quad (V, \nabla, \Fil)_{f+i}\cong (V, \nabla, \Fil)_{i}.\]
We simply represent this periodic Higgs-de Rham flow over $Y_1$ with the following diagram
\begin{equation} \label{diag:PHDF}
\xymatrix@W=10mm@C=3mm@R=10mm{
& (V, \nabla, \Fil)_{0} \ar[dr]^{\Gr}
&& (V, \nabla, \Fil)_{1} \ar[dr]^{\Gr}
&& (V, \nabla, \Fil)_{f-1} \ar[dr]^{\Gr}
&\\
(E, \theta)_{0} \ar[ur]^{\mathcal{C}_1^{-1}}
&& (E, \theta)_{1} \ar[ur]^{\mathcal{C}_1^{-1}}
&& \cdots \ar[ur]^{\mathcal{C}_1^{-1}}
&& (E, \theta)_{f} \ar@/^30pt/[llllll]^{\Phi}\\}
\end{equation}
\begin{theorem}[Lan-Sheng-Zuo \cite{LSZ19}]\label{equiv:FF&HDF}
There exists an equivalence between the category of strict $p^n$-torsion logarithmic Fontaine-Faltings modules with $W_n(\mathbb F_{p^f})$-endomorphism structure over $Y_n$ and the category of $f$-periodic  logarithmic Higgs-de Rham flows over $Y_n$.
\end{theorem}

\subsection{Parabolic Higgs-de Rham flows}
To define parabolic Higgs-de Rham flows, we need the parabolic versions of the categories and functors appearing in a diagram from \cite[Section 1.2.1]{SYZ22}. The most crucial part is to define the parabolic inverse Cartier functor.

\begin{definition}
    \begin{itemize}
        \item Denote by \emph{$\mathcal{H}_{p-2}^{par}(Y_n)$} the category of tuples $(E, \theta, \overline{V}, \overline{\nabla}, \overline{\Fil}, {\psi})$, each consisting of
\begin{itemize}
\item a graded parabolic Higgs bundle $(E, \theta)$ over $(Y_n, D_{Y_n})/S_n$ with exponent $\leq p-2$,
\item a filtered parabolic de Rham bundle over $(Y_n, D_{Y_{n-1}})/S_{n-1}$, and
\item an isomorphism of graded parabolic Higgs bundles over $(Y_{n-1}, D_{Y_{n-1}})/S_{n-1}$
\[{\psi}\colon \Gr(\overline{V}, \overline{\nabla}, \overline{\Fil}) \cong (E, \theta)\otimes \mathbb Z/p^{n-1}\mathbb Z.\]
\end{itemize}
\item Denote by \emph{$\mathcal{H}_{p-2}^{ni,par}(Y_n)$} the category whose objects satisfy the same conditions as those of  $\mathcal{H}_{p-2}^{par}(Y_n)$, except that the integrability condition is not imposed.
    \end{itemize}

\end{definition}

Following Lan-Sheng-Zuo's method, we construct the parabolic inverse functor as shown in the diagram:
\begin{equation}\label{diag:C^{-1}}
\xymatrix{
& \mathcal{H}^{ni,par}_{p-2}(Y_n)
&\mathcal{H}^{par}_{p-2}(Y_n)\ar[dr]^-{\mathcal C_n^{-1}}  \ar@{^(->}[l]\ar[dr]^-{\mathcal C_n^{-1}} \ar[dd]^{\mathcal{T}_n}&\\
\MCF^{par}_{p-2}(Y_n)\ar[dr]_{\widetilde{(\cdot)}}\ar[ur]^{\overline{\Gr}}
&&& \MIC^{par}(Y_n)\\
&\widetilde{\MC}^{par}(Y_n) &\widetilde{\MIC}^{par}(Y_n) \ar@{^(->}[l] \ar[ur]_-{\mathcal F_n}&\\
}
\end{equation}

\subsubsection{Functor $\overline{\Gr}$} Just as in the non-parabolic case, for an object $(V, \nabla, \Fil)$ in $\MCF^{par}_{p-2}(Y_n)$, the functor $\overline{\Gr}$ is given by
\[\overline{\Gr}(V, \nabla, \Fil)=(E, \theta, \overline{V}, \overline{\nabla}, \overline{\Fil}, {\psi}),\]
where $(E, \theta)=\Gr(V, \nabla, \Fil)$ is the graded parabolic Higgs bundle, $(\overline{V}, \overline{\nabla}, \overline{\Fil})$ is the modulo $p^{n-1}$-reduction of $(V, \nabla, \Fil)$, and $\psi$ is the identifying map
\[\id\colon \Gr(\overline{V}, \overline{\nabla}, \overline{\Fil})= (E, \theta)\otimes\mathbb Z/p^{n-1}\mathbb Z.\]

\subsubsection{The construction of functor $\mathcal{T}_n$}
We first recall the construction of $\mathcal{T}_n$ in non-parabolic case in \cite{LSZ19}. Denote by
\emph{$\widetilde{\MIC}(Y_n)$} (reps. \emph{$\MCF_{p-2}(Y_n)$}, \emph{$\mathcal H_{p-2}^{ni}(Y_n)$}, \emph{$\mathcal H_{p-2}(Y_n)$}) the full subcategory of $\widetilde{\MIC}^{par}(Y_n)$
(reps. $\MCF^{par}_{p-2}(Y_n)$, $\mathcal H^{ni,par}_{p-2}(Y_n)$, $\mathcal H_{p-2}^{par}(Y_n)$) consisting of objects with trivial parabolic structure.

Let $(E, \theta, \bar{V}, \bar{\nabla}, \overline{\Fil}, \psi)\in \mathcal{H}_{p-2}(Y_n).$
 By \cite[Lemma 4.6]{LSZ19}, the grading  functor
\[\overline{\Gr}:\MCF_{p-2}(Y_n)\to\mathcal H^{ni}_{p-2}(Y_n)\] locally essentially surjective. Locally on $Y_n$, choose an object  $(V,\nabla,\Fil)$ in $\MCF_{p-2}(Y_n)$ such that
\[\overline{\Gr}(V,\nabla,\Fil)\cong (E, \theta, \bar{V}, \bar{\nabla}, \overline{\Fil}, \psi).\]
Since $(E, \theta, \bar{V}, \bar{\nabla}, \overline{\Fil}, \psi)$ is integrable, $\widetilde{(V,\nabla,\Fil)}$ remains integrable even if $(V,\nabla,\Fil)$ is not necessarily so (\cite[Lemma 4.7]{LSZ19}), i.e., $\widetilde{(V,\nabla,\Fil)}$ is a locally parabolic $p$-flat bundle on $Y_n$.
By \cite[Lemma 4.10]{LSZ19}, the locally parabolic $p$-flat bundles glue to be a global object in $\widetilde{\MIC}(Y_n)$, which is still denoted by $\mathcal T_n(E, \theta, \bar{V}, \bar{\nabla}, \overline{\Fil}, \psi)$. Specifically, for any $(V,\nabla,\Fil)$ in $\MIC_{p-2}(Y_n)$, one has a canonical isomorphism
\[\mathcal T_n\circ\overline{\Gr} (V,\nabla,\Fil)\cong \widetilde{(V,\nabla,\Fil)}.\]

As a generalization of Lan-Sheng-Zuo's construction, we construct the functor $\mathcal{T}_n$ in the parabolic setting. Let $(E, \theta, \bar{V}, \bar{\nabla}, \overline{\Fil}, \psi)$ be an object in $\mathcal{H}^{par}_{p-2}(Y_n)$. For any $\alpha\in\mathbb{Q}^n$, denote
\[(\widetilde{V}_\alpha, \widetilde{\nabla}_\alpha)\coloneqq\mathcal{T}_n(E_\alpha, \theta_\alpha, \bar{V}_\alpha, \bar{\nabla}_\alpha, \overline{\Fil}_\alpha, \psi_\alpha),\]
By the following lemma, $\{(\widetilde{V}_\alpha, \widetilde{\nabla}_\alpha)\}_{\alpha\in \mathbb Q^n}$ forms a parabolic $p$-flat bundle. Then we can define $\mathcal{T}_n$ in the parabolic case by setting
\[\mathcal{T}_n(E, \theta, \bar{V}, \bar{\nabla}, \overline{\Fil}, \psi)\coloneqq \{(\widetilde{V}_\alpha, \widetilde{\nabla}_\alpha)\}_{\alpha\in \mathbb Q^n}.\]
$\mathcal T_n$ in non-parabolic case is regarded as a special case of our construction.

\begin{lemma}
The collection $\{(\widetilde{V}_\alpha, \widetilde{\nabla}_\alpha)\}_{\alpha\in \mathbb Q^n}$ forms a parabolic $p$-flat bundle.
We denote it by
\[\mathcal{T}_n(E, \theta, \bar{V}, \bar{\nabla}, \overline{\Fil}, \psi)\coloneqq\{(\widetilde{V}_\alpha, \widetilde{\nabla}_\alpha)\}_{\alpha\in \mathbb Q^n}.\]
\end{lemma}
\begin{proof}
   By the functoriality of $\mathcal T_n$, it suffices to verify that $\{\widetilde V_{\alpha}\}_{\alpha\in \mathbb Q^n}$ is a parabolic bundle.
     Since the problem is local on $Y_n$, we can assume that $\overline{\Fil}^k$ ($k=0,1,\ldots, p-2$) is isomorphic direct sum of parabolic line bundles and splitting.
     Let $E=\oplus_{i=0}^{p-2} E^i$  be the grading structure of $E$. Denote by
     \[V\coloneqq E, \quad \Fil^k\coloneqq \oplus_{i=k}^{p-2} E^i.\]
     Then $\overline{V}=V\mod p^{n-1}$ and $\overline{\Fil}^k=\Fil^k\mod p^{n-1}$. Using the same method as in \cite[Lemma 4.6]{LSZ19}, one can choose a connection $\nabla_0$ on $V_0$ satisfying the Griffiths transversality with respect to $\Fil_0$ and
     \[\overline{\Gr}(V_0,\nabla_0,\Fil_0)=(E_0,\theta_0,\overline{V}_0,\overline{\nabla}_0,\overline{\Fil}_0,\psi_0).\]
     By \autoref{thm_UnqiuePara}, $\nabla_0$ extends to a unique parabolic connection $\nabla$ on $V$ satisfying the Griffiths transversality with respect to $\Fil$ and
     \[\overline{\Gr}(V,\nabla,\Fil)=(E,\theta,\overline{V},\overline{\nabla},\overline{\Fil},\psi).\]
     $(V,\nabla,\Fil)$ is an object in $\MCF^{par}_{p-2}(Y_n)$. By \autoref{Prop_FaltingTilde},
     $\{\widetilde V_{\alpha}\}_{\alpha\in\mathbb Q^n}=\widetilde{(V,\Fil)}$ is a parabolic vector bundle.
\end{proof}

The composition functor $\mathcal C_n^{-1}:=\mathcal F_n\circ \mathcal{T}_n: \mathcal H^{par}_{p-2}(Y_n)\longrightarrow \MIC^{par}(Y_n)$ is called \emph{the parabolic inverse Cartier functor of $n$-th truncated level}. By the construction of $\mathcal C_n^{-1}$, the restriction of $\mathcal C_n^{-1}$ on the category $\mathcal H_{p-2}(Y_n)$ coincides with the usual inverse Cartier functor  in \cite[section 4]{LSZ19}.

\subsubsection{Parabolic Higgs-de Rham flows and parabolic Fontaine-Faltings modules}

\begin{definition}
A \emph{parabolic Higgs-de Rham flow} over $(Y_n, D_{Y_n}) \subset (Y_{n+1}, D_{Y_{n+1}})$ is a sequence consisting of infinitely many alternating terms of filtered parabolic de Rham bundles and graded parabolic Higgs bundles
\[\left\{
(V, \nabla, \Fil)^{(n-1)}_{-1},
(E, \theta)_{0}^{(n)},
(V, \nabla, \Fil)_{0}^{(n)},
(E, \theta)_{1}^{(n)},
(V, \nabla, \Fil)_{1}^{(n)},
\cdots\right\},\]
which are related to each other by the following diagram inductively
\begin{equation*}
\xymatrix@W=5mm@C=-3mm@R=5mm{
&& (V, \nabla, \Fil)_{0}^{(n)} \ar[dr]^{\Gr}
&& (V, \nabla, \Fil)_{1}^{(n)} \ar[dr]^{\Gr}
&& (V, \nabla, \Fil)_{2}^{(n)} \ar[dr]^{\Gr}
&\\
& (E, \theta)_{0}^{(n)} \ar[ur]^{\mathcal{C}^{-1}_n} \ar@{..>}[dd]
&& (E, \theta)_{1}^{(n)} \ar[ur]^{\mathcal{C}^{-1}_n}
&& (E, \theta)_{2}^{(n)} \ar[ur]^{\mathcal{C}^{-1}_n}
&& \cdots\\
(V, \nabla, \Fil)^{(n-1)}_{-1} \ar[dr]^{\Gr}
&&&&&&&\\
&\Gr\left((V, \nabla, \Fil)^{(n-1)}_{-1}\right)
&&&&\\
}
\end{equation*}
where
\begin{itemize}
\item[-] $(V, \nabla, \Fil)^{(n-1)}_{-1}\in \MCF^{par}_{p-2}(Y_n)$;
\item[-] $(E, \theta)_0^{(n)}$ is a lifting of the graded parabolic Higgs bundle $\Gr\left((V, \nabla, \Fil)^{(n-1)}_{-1}\right)$ over $(Y_n, D_{Y_n})/S_n$, $(V, \nabla)_0^{(n)}\coloneqq \mathcal C^{-1}_n ((E, \theta)_0^{(n)}, (V, \nabla, \Fil)^{(n-1)}_{-1} , \psi)$, and $\Fil^{(n)}_0$ is a parabolic Hodge filtration on $(V, \nabla)_0^{(n)}$ of level in $[0, p-2]$;
\item[-] Inductively, for $m\geq1$, $(E, \theta)_m^{(n)}\coloneqq\Gr\left((V, \nabla, \Fil)_{m-1}^{(n)}\right)$ and
\[(V, \nabla)_m^{(n)}\coloneqq \mathcal C^{-1}_n \left(
(E, \theta)_{m}^{(n)},
(V, \nabla, \Fil)^{(n-1)}_{m-1},
\id
\right).\]
Here, $(V, \nabla, \Fil)^{(n-1)}_{m-1}$ is the reduction of $(V, \nabla, \Fil)^{(n)}_{m-1}$ on $X_{n-1}$. And $\Fil^{(n)}_m$ is a Hodge filtration on $(V, \nabla)_m^{(n)}$.
\end{itemize}
\end{definition}

\begin{definition}
Let
\[\Flow = \left\{
(V, \nabla, \Fil)^{(n-1)}_{-1},
(E, \theta)_{0}^{(n)},
(V, \nabla, \Fil)_{0}^{(n)},
(E, \theta)_{1}^{(n)},
(V, \nabla, \Fil)_{1}^{(n)},
\cdots\right\},\]
and
\[\Flow'=\left\{
(V', \nabla', \Fil')^{(n-1)}_{-1},
(E', \theta')_{0}^{(n)},
(V', \nabla', \Fil')_{0}^{(n)},
(E', \theta')_{1}^{(n)},
(V', \nabla', \Fil')_{1}^{(n)},
\cdots\right\},\]
be two Higgs-de Rham flows over $(Y_n, D_{Y_n})\subset (Y_{n+1}, D_{Y_{n+1}})$. A morphism from $\Flow$ to $\Flow'$ is a compatible system of morphisms
\[\{\varphi_{-1}^{(n-1)}, \psi_{0}^{(n)}, \varphi_{0}^{(n)}, \psi_{1}^{(n)}, \varphi_{1}^{(n)}, \cdots\}\]
between the terms respectively in the following sense:
\begin{itemize}
\item $\Gr(\varphi_{-1}^{(n-1)})=\psi_{0}^{(n)}\pmod{p^{n-1}}$;
\item $\mathcal C_n^{-1}(\psi_m^{(n)}, \varphi_{m-1}^{(n-1)}) = \varphi_m^{(n)}$ (if $m\geq1$, then here $\varphi_m^{(n-1)}:=\varphi_m^{(n)}\pmod{p^{n-1}}$); and
\item $\Gr(\varphi_m^{(n)})=\varphi_{m+1}^{(n)}$.
\end{itemize}
\end{definition}

\begin{remark}
The morphism is uniquely determined by the first two terms $\varphi_{-1}^{(n-1)}, \psi_{0}^{(n)}$. If $n=1$, then the first morphism is vacuous.
\end{remark}

\begin{definition}
Let
\[\Flow = \left\{
(V, \nabla, \Fil)^{(n-1)}_{-1},
(E, \theta)_{0}^{(n)},
(V, \nabla, \Fil)_{0}^{(n)},
(E, \theta)_{1}^{(n)},
(V, \nabla, \Fil)_{1}^{(n)},
\cdots\right\},\]
be a Higgs-de Rham flow. We call the flow
\[\Flow[f] = \left\{
(V, \nabla, \Fil)^{(n-1)}_{f-1},
(E, \theta)_{f}^{(n)},
(V, \nabla, \Fil)_{f}^{(n)},
(E, \theta)_{f+1}^{(n)},
(V, \nabla, \Fil)_{f+1}^{(n)},
\cdots\right\}\]
the \emph{$f$-th shifting of $\Flow$}, where $(V, \nabla, \Fil)^{(n-1)}_{f-1}:=(V, \nabla, \Fil)^{(n)}_{f-1}\pmod{p^{n-1}}$.
\end{definition}

\begin{definition}
If there is an isomorphism
\[\psi:=\{\varphi_{-1}^{(n-1)}, \psi_{0}^{(n)}, \varphi_{0}^{(n)}, \psi_{1}^{(n)}, \varphi_{1}^{(n)}, \cdots\}\]
from $\Flow[f]$ to $\Flow$, then we call the pair $(\Flow, \psi)$ a \emph{periodic Higgs-de Rham flow}. Note that $\psi$ is part of the data in the periodic Higgs-de Rham flow, which we will call a periodic mapping of $\Flow$.
Since $\psi$ is uniquely determined by $\varphi_{-1}^{(n-1)}, \psi_{0}^{(n)}$, we sometimes use $(\Flow, (\varphi_{-1}^{(n-1)}, \psi_{0}^{(n)}))$ to represent $(\Flow, \psi)$. In the case $n=1$, the first term in the flow is vacuous, and we also use $(\Flow, \psi_{0}^{(1)})$ to represent $(\Flow, \psi)$.
\end{definition}

We generalize \cite[Theorem 5.3]{LSZ19} to the parabolic version. We first prove the 1-periodic case.

\begin{lemma}\label{lem_1-HiggsCorrespondence}
    There is an equivalence of categories from the category of $1$-periodic parabolic Higgs-de Rham flows over $(Y_n,D_{Y_n})/S_n$ to the category $\MF_{[0, p-2]}^\nabla((Y_n, D_{Y_n})/S_n)$.
\end{lemma}

\begin{proof}
    By the definition of parabolic Higgs-de Rham flows, to give a $1$-periodic parabolic Higgs-de Rham flows over $(Y_n,D_{Y_n})/S_n$ is equivalent to give a tuple $(V,\nabla,\Fil,\psi)$, where $(V,\nabla,\Fil)$ is an object in $\MIC^{par}_{p-2}(Y_n)$ and
    \[\psi:\mathcal C_n^{-1}\circ\overline{\Gr}(V,\nabla,\Fil)\longrightarrow (V,\nabla)\]
    is an isomorphism of parabolic de Rham bundles. To prove the equivalence of categories, it is sufficient to give a canonical isomorphism
    \[\mathcal C_n^{-1}\circ \overline{\Gr}(V,\nabla,\Fil)\cong \mathcal F_n\widetilde{(V,\nabla,\Fil)}.\]
    Since $\mathcal C_n^{-1}=\mathcal F_n\circ \mathcal T_n$, it reduce to prove
    \[\mathcal T_n\circ \overline{\Gr}(V,\nabla,\Fil)\cong \widetilde{(V,\nabla,\Fil)},\] which follows from
    $\mathcal T_n\circ \overline{\Gr}(V_{\alpha},\nabla_{\alpha},\Fil_{\alpha})\cong \widetilde{(V_{\alpha},\nabla_{\alpha},\Fil_{\alpha})}$ for any $\alpha \in \mathbb Q^n$.
\end{proof}

\begin{theorem} \label{thm_equFunctorHdRF&FFMod}
We keep the same notation as in \autoref{setup_PFF and PHDF}. Then there is an equivalence of categories from the category of $f$-periodic parabolic Higgs-de Rham flows over $(Y_n,D_{Y_n})/S_n$ to the category $\MF_{[0, p-2],f}^\nabla((Y_n, D_{Y_n})/S_n)$.
\end{theorem}
\begin{proof} Given a $f$-periodic Higgs-de Rham flows
\[\Flow = \left\{
(V, \nabla, \Fil)^{(n-1)}_{-1},
(E, \theta)_{0}^{(n)},
(V, \nabla, \Fil)_{0}^{(n)},
(E, \theta)_{1}^{(n)},
(V, \nabla, \Fil)_{1}^{(n)},
\cdots\right\}\]
over $(Y_n,D_{Y_n})/S_n$, let
\[(V,\nabla,\Fil)\coloneqq \bigoplus_{i=0}^{f-1}(V,\nabla,\Fil)^{(n)}_{i}.\]
Then by the definition of Higgs-de Rham flows, one has
\[\mathcal C_n^{-1}\circ \overline{\Gr}(V,\nabla,\Fil)\cong \bigoplus_{i=1}^{f}(V,\nabla,\Fil)^{(n)}_{i}.\]
Since $\Flow$ is $f$-periodic, the isomorphism of Higgs-de Rham flows $\Flow[f]\cong \Flow$ induces an isomorphism
\[\psi:\mathcal C_n^{-1}\circ \overline{\Gr}(V,\nabla,\Fil)\cong (V,\nabla,\Fil).\]
Thus the tuple $(V,\nabla,\Fil,\psi)$ is a $1$-periodic parabolic Higgs-de Rham flows, which gives a strict $p^n$-torsion parabolic Fontaine-Faltings modules $(V,\nabla,\Fil,\varphi)$ over $(Y_n,D_{Y_n})/S_n$ by \autoref{lem_1-HiggsCorrespondence}. Let $\xi$ be a generator of $W_n(\mathbb F_{p^f})$ over $\mathbb Z/p^n\mathbb Z$ and let $\sigma$ be the Frobenius on $W_n(\mathbb F_{p^f})$. By setting $\iota(\xi)$ to be the map $m_{\xi}\oplus m_{\sigma(\xi)}\oplus\cdots\oplus m_{\sigma^{f-1}(\xi)}$, where $m_{\sigma^{i}(\xi)}$ is the multiplication map by $\sigma^{i}(\xi)$ on $(V,\nabla,\Fil)_i^{(n)}$, we can define a $W_n(\mathbb F_{p^f})$-endomorphism structure
\[\iota: W_n(\mathbb F_{p^f}) \to \End (V,\nabla,\Fil,\varphi).\]

Conversely, given a strict $p^n$-torsion parabolic Fontaine-Faltings modules with $W_n(\mathbb F_{p^f})$-endomorphism structure $(V,\nabla,\Fil,\varphi,\iota)$, let
\[(V,\nabla,\Fil)_i^{(n)}\coloneqq (V,\nabla,\Fil)^{\iota(\xi)=\sigma^i(\xi)}\]
be the $\sigma^i(\xi)$-eigenspace of $\iota(\xi)$. Then by comparing the $\sigma^{i+1}(\xi)$-eigenspaces on the both side of the isomorphism
of $\varphi$, one gets isomorphisms
\[\varphi_i:\mathcal C_n^{-1}\circ \overline{\Gr}(V,\nabla,\Fil)_i^{(n)}\cong (V,\nabla,\Fil)_{i+1}^{(n)}\quad \text{ for } 0\geq i\geq f-2\]
and
\[\varphi_{f-1}:\mathcal C_n^{-1}\circ \overline{\Gr}(V,\nabla,\Fil)_{f-1}^{(n)}\cong (V,\nabla,\Fil)_{0}^{(n)}.\]
By setting
\[(V,\nabla,\Fil)^{(n-1)}_{-1}\coloneqq (V,\nabla,\Fil)^{(n)}_{f-1} \mod p^{n-1},\]
we defines a $f$-periodic Higgs-de Rham flows over $(Y_n,D_{Y_n})/S_n$. Obviously, the two processes are quasi-inverse to each other, thus inducing an equivalence of categories.
\end{proof}
\begin{remark}
When restricted to the open subset $Y_n\setminus D_{Y_n}$, the equivalence of categories in \autoref{thm_equFunctorHdRF&FFMod} coincides with that in \cite[Theorem 5.3]{LSZ19}.
\end{remark}

\subsubsection{Comparing with Krishnamoorthy-Sheng's work} In \cite{KrSh20}, Krishnamoorthy-Sheng established the theory of parabolic Higgs-de Rham flows over $(Y_1,D_{Y_1})/k$. They constructed the parabolic inverse Cartier functor on $(Y_1,D_{Y_1})/k$ as follows.

Let $f:(Y'_1, D_{Y'_1})\to (Y_1, D_{Y_1})$ be a cyclic cover of order $N$ with branch divisor $D_{Y_1}$ and let $G=\mathbb Z/N\mathbb Z$. Then
\[f^*D_{Y_1}=ND_{Y'_1}.\] There exists a pair of functors $(f_{par}^*,f_{par*})$ which is quasi-inverse to each other and induced an equivalence between the category of parabolic $\lambda$-flat bundle on $(Y_1,D_{Y_1})$ whose parabolic weights are contained in $\frac{1}{N}\mathbb Z$ and the category of $G$-equivariant $\lambda$-flat bundle on $(Y_1,D_{Y_1})$ (\cite[Proposition 2.14]{KrSh20}).
For any parabolic Higgs bundle $(E, \theta)$ over $(Y_1, D_{Y_1})/k$ whose weights are contained in $\mathbb{Z}/N\mathbb{Z}$, the parabolic inverse Cartier transform for $(E, \theta)$ is defined to be $f_{par*}\circ \mathcal C_{1,tirv}^{-1}\circ f_{par}^*(E, \theta)$, where $\mathcal C_{1,tirv}^{-1}$ is the inverse Cartier functor for usual vector bundles (with trivial parabolic structure) over $(Y_1,D_{Y_1})/k$.

In what follows, we will show that there is an isomorphism of functors between $\mathcal C_{1,triv}^{-1}\circ f^*$ and $f^*\circ \mathcal C_1^{-1}$. Since we have show $f_{par}^*\cong f^*$ in this case and $(f_{par}^*,f_{par*})$ are quasi-inverse to each other, then
\[f_{par*}\circ \mathcal C_{1,tirv}^{-1}\circ f_{par}^*\cong \mathcal C_1^{-1}.\]
This implies that the parabolic inverse Cartier functor constructed by Krishnamoorthy-Sheng coincides with $\mathcal C_1^{-1}$, hence their definition of parabolic Higgs-de Rham flows is a special case of ours.

Let $\widetilde{f}:(\widetilde{Y'_1}, D_{\widetilde{Y'_1}})\to (\widetilde{Y_1}, D_{\widetilde{Y}_1})$ be a $W_2$-lifting of $f$. Denote by $\Phi$ and $\Phi'$ the absolute Frobenius of $Y_1$ and $Y'_1$. Choose an open covering $\{\widetilde{U}_{\mu}\}_{\mu\in \Lambda}$ of $\widetilde{Y}_1$ and a family of morphisms $\{\widetilde{\Phi}_{\mu}:\widetilde{U}_{\mu}\to \widetilde{U}_{\mu}\}_{\mu\in \Lambda}$ such that $\widetilde{\Phi}_{\mu}$ lifts $\Phi\mid_{U_\mu}$. Let
\[U_{\mu}=\widetilde{U}_{\mu}\otimes \mathbb{Z}/p\mathbb{Z}, \quad U'_{\mu}=f^{-1}(U_{\mu}), \quad \widetilde{U'}_{\mu}=f^{-1}(\widetilde{U}_{\mu}).\]
Let $\widetilde{\Phi'}_{\mu}:\widetilde{U'}_\mu\to \widetilde{U'}_\mu$ be a lifting of $\Phi'\mid_{U'_\mu}$.
Consider the $W_2$-morphism $\varphi_{\mu}:\widetilde{V}_{\mu}\to \widetilde{U'_\mu}$ sitting in the following Cartesian diagram:
\begin{equation}\label{diag_FrobCartDiag}
\xymatrix{ \widetilde{V}_{\mu}\ar[r]^{\varphi_{\mu}}\ar[d]_{g_\mu}&\widetilde{U'_\mu}\ar[d]^{\widetilde{f}}\\
\widetilde{U}_{\mu}\ar[r]_{\widetilde{\Phi}_{\mu}}&\widetilde{U}_{\mu}
}.
\end{equation}
$\widetilde{V}_{\mu}$ is a lifting of $U'_\mu$ and $\varphi_\mu$ is a lifting of $\Phi'_\mu\mid_{U'_\mu}$. Then by \cite[Lemma 5.4]{DeIl87}, for any $\mu,\tau\in\Lambda$ there are canonical homomorphisms
\begin{equation*}
\begin{split}
h_\mu:&\Phi'^*\Omega_{U'_{\mu}/S}^1(\log D_{Y'}\cap U'_{\mu})\to\mathcal{O}_{U'_{\mu}}, \\
h_{\mu\tau}:&\Phi'^*\Omega_{U'_{\mu\tau}/S}^1(\log D_{Y'}\cap U'_{\mu\tau})\to\mathcal{O}_{U'_{\mu\tau}}, \\
h'_{\mu\tau}:&\Phi'^*\Omega_{U'_{\mu\tau}/S}^1(\log D_{Y'}\cap U'_{\mu\tau})\to\mathcal{O}_{U'_{\mu\tau}}, \\
\end{split}
\end{equation*}
such that
\begin{equation*}
\begin{split}
\rmd h_\mu=&\frac{\rmd \varphi_\mu}{p}-\frac{\rmd \widetilde{\Phi'}_\mu}{p}, \\
\rmd h_{\mu\tau}=&\frac{\rmd \widetilde{\Phi'}_\mu}{p}-\frac{\rmd \widetilde{\Phi'}_\tau}{p}, \\
\rmd h'_{\mu\tau}=&\frac{\rmd \varphi_\mu}{p}-\frac{\rmd \varphi_\tau}{p}, \\
\end{split}
\end{equation*}
and satisfy the cocycle condition
\begin{equation}\label{CocycleCondition}
h_\mu+h_{\mu\tau}=h_{\tau}+h'_{\mu\tau}.
\end{equation}
Let
\[(E_\mu, \theta_\mu)\coloneqq (E, \theta)\mid_{U_\mu}, \quad (E', \theta')\coloneqq f^*(E, \theta), \quad (E'_\mu, \theta'_\mu)\coloneqq (E', \theta')\mid_{U'_\mu}.\]
$\mathcal C_1^{-1}(E, \theta)$ is locally defined to be
$\left(\Phi^*E_\mu, \nabla_{can}+\frac{\rmd \widetilde{\Phi}_{\mu}}{p}\right)$ and glued by
\[\exp(h_{\mu\tau}(\Phi^*\theta)): \Phi^*E_{\mu\tau}\to \Phi^*E_{\mu\tau}.\]
 By the commutativity of the diagram \ref{diag_FrobCartDiag}, we have
\[f^* \frac{\rmd \widetilde{\Phi}_{\mu}}{p}=\frac{\rmd \varphi_{\mu}}{p}, \quad f^*h_{\mu\tau}=h'_{\mu\tau}.\]
 Thus $f^*\circ \mathcal C_1^{-1}(E, \theta)$ is isomorphic to the parabolic de Rham bundle which is locally isomorphic to $\left(\Phi'^*E'_\mu, \nabla_{can}+\frac{\rmd \varphi_\mu}{p}\right)$ and glued by
\[G'_{\mu\tau} \coloneqq \exp(h'_{\mu\tau}(\Phi'^*\theta')): \Phi'^*E'_{\mu\tau}\to \Phi'^*E'_{\mu\tau}.\]
\begin{proposition}\label{thm_CompareInverseCartierFunctor}
Let $f:(Y'_1, D_{Y'_1})\to (Y_1, D_{Y_1})$ be a cyclic cover of order $N$ with branch divisor $D_{Y_1}$. Then for any parabolic Higgs bundle $(E, \theta)$ over $(Y_1, D_{Y_1})/S$ whose parabolic weights are contained in $\frac{1}{N}\mathbb Z$, there is a natural isomorphism of de Rham bundle over $(Y'_1, D_{Y'_1})/S$
\[f^*\mathcal C_1^{-1}(E, \theta)\cong \mathcal C_{1,triv}^{-1}f^*(E, \theta),\]
where $C_{1,triv}^{-1}$ is the inverse Cartier functor for usual Higgs bundles. This implies that the parabolic inverse Cartier functor constructed in \cite{KrSh20} coincides with $\mathcal C_1^{-1}$.
\end{proposition}
\begin{proof}
The verification details are contained in the following steps.

\textit{Step 1. Reduce to usual Higgs bundles.} If the proposition holds for usual Higgs bundles, then for any parabolic Higgs bundle $(E, \theta)$, there is an isomorphism of parabolic de Rham bundles:
\[\psi:f^*\mathcal C_1^{-1}(E_0(D_Y), \theta_0(D_Y))\cong \mathcal C_{1,triv}^{-1}f^*(E_0(D_Y), \theta_0(D_Y)).\]
Considering $f^*\mathcal C_1^{-1}(E, \theta)$ as a sub parabolic de Rham bundle of $f^*\mathcal C_1^{-1}(E_0(D_Y), \theta_0(D_Y))$, it sufficient to show
\[\psi\big(f^*\mathcal C_1^{-1}(E, \theta)\big)=\mathcal C_{1,triv}^{-1}f^*(E, \theta).\]
since there is a decomposition of parabolic Higgs bundles:
\[(E, \theta)= \bigcup_{\gamma\in\mathbb{Q}^n}\big(E(-\gamma D)_0, \theta(-\gamma D)_0\big) \otimes \big(\mathcal{O}_Y(\gamma D), 0\big),\]
we see that
\begin{equation*}
\begin{split}
\psi\big(f^*\mathcal C_1^{-1}(E, \theta)\big)&=\bigcup_{\gamma\in\mathbb{Q}^n}\psi\Big(f^*\mathcal C_1^{-1}\big(E(-\gamma D)_0, \theta(-\gamma D)_0\big)\Big)\otimes \psi\Big(f^*\mathcal C_1^{-1}\big(\mathcal{O}_Y(\gamma D), 0\big)\Big)\\
&=\bigcup_{\gamma\in\mathbb{Q}^n}\mathcal C_{1,triv}^{-1}f^*\big(E(-\gamma D)_0, \theta(-\gamma D)_0\big)\otimes \Big(\mathcal{O}_{Y'}\big(pf^*(\gamma D_Y)\big), \rmd\big(pf^*(\gamma D_Y)\big)\Big)\\
&=\bigcup_{\gamma\in\mathbb{Q}^n}\mathcal C_{1,triv}^{-1}f^*\big(E(-\gamma D)_0, \theta(-\gamma D)_0\big)\otimes \mathcal C_1^{-1}f^*\big(\mathcal{O}_Y(\gamma D), 0\big)\\
&=\mathcal C_{1,triv}^{-1}f^*(E, \theta).
\end{split}
\end{equation*}
The proposition follows. In the rest part of the proof, since $\mathcal C_{1,triv}^{-1}$ coincides with $\mathcal C_1^{-1}$ for Higgs bundles with trivial parabolic structure, we simply write $\mathcal C_1^{-1}$ for $\mathcal C_{1,triv}^{-1}$.

\textit{Step 2. Local isomorphism.} Assume $(E, \theta)$ is a usual Higgs bundle. By the construction of the inverse Cartier functor, $\mathcal C_1^{-1}\circ f^*(E, \theta)$ is locally isomorphic to
\[\left(\Phi'^*E'_\mu, \nabla_{can}+\frac{\rmd \widetilde{\Phi'}_{\mu}}{p}\right)\]
and glued by
\[G_{\mu\tau}\coloneqq \exp(h_{\mu\tau}(\Phi'^*\theta')):\Phi'^*E'_{\mu\tau}\to \Phi'^*E'_{\mu\tau}.\]
By the discussion above, $f^*\circ \mathcal C_1^{-1}(E, \theta)$ is locally isomorphic to
\[\left(\Phi'^*E'_\mu, \nabla_{can}+\frac{\rmd \varphi_{\mu}}{p}\right).\]
Then
\[\psi_\mu\coloneqq\exp(h_\mu(\Phi'^*\theta')):\left(\Phi'^*E'_\mu, \nabla_{can}+\frac{\rmd \varphi_{\mu}}{p}\right)\to \left(\Phi'^*E'_\mu, \nabla_{can}+\frac{\rmd \widetilde{\Phi'}_{\mu}}{p}\right)
\]
is a local isomorphism between $f^*\circ \mathcal C_1^{-1}(E, \theta)$ and $\mathcal C_1^{-1}\circ f^*(E, \theta)$.


\textit{Step 3. Gluing isomorphisms.} We need to show that the local isomorphisms glue to form a global isomorphism. By the cocycle condition \ref{CocycleCondition} and the integrability of $\theta$,
\begin{equation*}
\begin{split}
G_{\mu\tau}\circ \psi_\mu&=\exp(h_{\mu\tau}(\Phi'^*\theta')) \circ \exp(h_{\mu}(\Phi'^*\theta'))\\
&=\exp((h_{\mu}+h_{\mu\tau})(\Phi'^*\theta'))\\
&=\exp(h_{\tau}(\Phi'^*\theta'))\circ\exp(h'_{\mu\tau}(\Phi'^*\theta'))\\
&=\psi_\tau\circ G'_{\mu\tau}.
\end{split}
\end{equation*}
Thus, the diagram
\begin{equation*}
\xymatrix{
\Phi'^*E'_{\mu\tau}\ar[r]^{\psi_\mu}\ar[d]_{G'_{\mu\tau}}&
\Phi'^*E'_{\mu\tau}\ar[d]^{G_{\mu\tau}}\\
\Phi'^*E'_{\mu\tau}\ar[r]_{\psi_{\tau}}&
\Phi'^*E'_{\mu\tau}\\
}
\end{equation*}
is commutative and the family of local isomorphisms $\{\psi_\mu\}_{\mu\in \Lambda}$ glues to form an isomorphism
\[\psi:f^*\mathcal C_1^{-1}(E, \theta)\cong \mathcal C_1^{-1}f^*(E, \theta).\qedhere\]
\end{proof}

\section{\bf An Application.}\label{sec_App}
In this section, as an application of the Higgs-de Rham flow theory of parabolic bundles, we construct a selfmap on the moduli space of parabolic Higgs bundles and determine the periodic parabolic Higgs bundles arising from families of elliptic curves.

\subsection{Selfmap on the moduli space of parabolic Higgs bundles} \label{sec_SelfmapParaHiggsBundle}
In \cite{SYZ22}, Sun, Yang and Zuo constructed a selfmap on the moduli space of Higgs bundles over the projective line, whose periodic points are associated with periodic Higgs bundles. By studying the properties of the selfmap, they constructed infinitely many twisted Fontaine-Faltings modules. In this subsection, we generalize the selfmap to the case of parabolic Higgs bundles.


Let $k$ be an algebraically closed field of characteristic $p>2$. Fix a $W_2(k)$-lifting of $\mathbb{P}^1_k$. Let $N \geq 2$ be a positive integer such that $p\nmid N$.
Let $M_{\Hig}^{N}$ be the moduli space of semistable graded parabolic Higgs bundles of rank $2$ degree 0 over $\mathbb{P}^1_k$ such that the Higgs fields have $m$ poles $\{x_1, x_2, \ldots, x_m\}$ and the parabolic weights are in the set
\[J_N\coloneqq\left\{\frac{k}{N}:k\in [1,N-1]\cap \mathbb Z,
 \gcd(k, N)=1\right\}\]
at $x_1$ and 0 at other points. Let $M_{\dR}^{N}$ be the moduli space of semistable graded parabolic de Rham bundles of rank $2$ degree 0 over $\mathbb{P}^1_k$ such that the connections have $m$ poles $\{x_1, x_2, \ldots, x_m\}$ and the parabolic weights are in the set $J_N$ at $x_1$ and $0$ at other points.

\subsubsection{Parabolic vector bundles over $\mathbb{P}^1_k$} We first classify the parabolic vector bundles on $\mathbb{P}^1_k$. Grothendieck's theorem shows that any vector bundle on $\mathbb{P}^1_k$ is isomorphic to a direct sum of line bundles. The theorem remains true for parabolic vector bundles. As a generalization of Grothendieck's theorem, we can prove that:
\begin{proposition}\label{thm_ParaGrothendieckThm}
Let $V$ be a parabolic bundle on $(\mathbb{P}^1_k, x_1)/k$ whose parabolic weights are contained in $\frac{1}{N}\mathbb{Z}$ with $p \nmid N$. Then $V$ is isomorphic to a direct sum of parabolic line bundles.
\end{proposition}
\begin{proof}
Without loss of generality, we can assume that $x_1=\infty$. Let
\[f:\mathbb{P}^1_{k}\to \mathbb{P}^1_{k}, \quad x\to x^N\]
 be the cyclic cover of order $N$ with branch divisor $0+\infty$ and group $G\cong \mathbb{Z}/N\mathbb{Z}$. Then $f^* V$ is a $G$-equivariant bundle. In \cite[Theorem 3.1]{BiFr24}, Biswas-Machu proved that any $G$-equivariant bundle on $\mathbb{P}^1_{\mathbb{C}}$ is isomorphic to a direct sum of $G$-equivariant line bundles when $G$ is a finite cyclic group. The same method can be applied to $\mathbb{P}^1_{k}$ when $\Char k\nmid N$ to prove that any $G$-equivariant vector bundle on $\mathbb{P}^1_{k}$ is isomorphic to a direct sum of $G$-equivariant line bundles. By \cite[Lemma 3.6]{Bis97}, there is a functor $f_{par*}$ from the category of $G$-equivariant vector bundles over $\mathbb{P}^1_{k}$  to the category of parabolic vector bundles over $\mathbb{P}^1_{k}$ whose parabolic weights are contained in $\frac{1}{N}\mathbb Z$, which is quasi-inverse to $f^*$.
 Thus $V\cong f_{par*}f^* V$ is isomorphic to a direct sum of parabolic line bundles.
\end{proof}

In the rest parts of the section, for any $\lambda \in \mathbb Q$, we denote by $$\mathcal O(\lambda)\coloneqq \mathcal O_{\mathbb P^1_k}(\lambda x_1).$$

\subsubsection{The self map on $M_{\Hig}^{N}$}
For any isomorphism class of Higgs bundles $[(E, \theta)]\in M_{\Hig}^{N}(k)$, $E$ is of the form $\mathcal{O}(\lambda)\oplus \mathcal{O}(\eta)$ with $\langle \lambda\rangle, \langle \eta\rangle \in J_N$ and $\lambda\geq \eta$ by \autoref{thm_ParaGrothendieckThm}. By the condition of degree 0, we have $\lambda=-\eta>0$ and by the condition that $(E, \theta)$ is semistable,
\[\theta\left(\mathcal{O}(\lambda)\right)\not\subseteq \mathcal{O}(\lambda)\otimes \Omega^1_{\mathbb{P}^1}(m).\]
 If not, the pair $(\mathcal{O}(\lambda), \theta\mid_{\mathcal{O}(\lambda)})$ is a sub parabolic Higgs bundle of $(E, \theta)$ with slope $\lambda>0$, which contradicts the semistability of $(E, \theta)$. Then the Higgs field
\[\theta: \mathcal{O}(\lambda)\to \mathcal{O}(-\lambda)\otimes \Omega^1_{\mathbb{P}^1}(m)\]
is nonzero in $\Hom\left(\mathcal{O}(\lambda), \mathcal{O}(-\lambda+m-2)\right)$, so we have $\lambda\leq \frac{m}{2}-1$ and $m\geq 3$. Thus $M_{\Hig}^{N}$ admits a decomposition
\[M_{\Hig}^{N}=\coprod_{\langle\lambda\rangle\in J_N, 0\leq \lambda\leq \frac{m}{2}-1} M_{\Hig}^{N, \lambda},\]
where $M_{\Hig}^{N, \lambda}$ is the sub moduli space of parabolic Higgs bundles of the form $\mathcal{O}(\lambda)\oplus\mathcal{O}(-\lambda)$,
\[M_{\Hig}^{N, \lambda}\cong \mathbb{P}\Hom\left(\mathcal{O}(\lambda), \mathcal{O}(-\lambda+m-2)\right)\cong \mathbb{P}^{[m-2-2\lambda]},\]
 and $\langle \lambda\rangle$ is the unique number in $[0, 1)$ such that $\lambda-\langle \lambda\rangle\in \mathbb{Z}$.

\[\mathcal{C}_1^{-1}: M_{\Hig}^{N}\to M_{\dR}^{N}.\]
In \cite{SYZ22}, Sun, Yang and Zuo introduced the notion of self map as follows: let $x=[(E, \theta)]\in M_{\Hig}^{N}(k)$ and $(V, \nabla)=\mathcal C_1^{-1}(E, \theta)$. Then $V\cong \mathcal{O}(\lambda_x)\oplus\mathcal{O}(-\lambda_x)$ for some $\lambda_x>0$ such that $\langle \lambda \rangle\in J_N$ by \cite[Proposition 2.18]{KrSh20}. By taking grading with respect to the filtration $\Fil$
\[0\subseteq \mathcal{O}(\lambda_x)\subseteq \mathcal{O}(\lambda_x)\oplus\mathcal{O}(-\lambda_x)=V,\]
we obtain a parabolic Higgs bundle whose isomorphism class is in $M_{\Hig}^{N, \lambda_x}(k)$. Let $\Gr$ be the grading functor with respect to $\Fil$. The self map is defined on $M_{\Hig}^{N}(k)$ by sending $[(E, \theta)]$ to $[\Gr\circ \mathcal C_1^{-1}(E, \theta)]$.

\subsubsection{Universal parabolic Higgs bundle}
Now we fix one of the maximal irreducible components $M_{\Hig}^{N, \frac{d}{N}}$ of $M_{\Hig}^{N}$ with $\frac{d}{N}\in J_N$ and $d<\frac{N}{2}$. For any $[(E, \theta)]\in M_{\Hig}^{N, \frac{d}{N}}(k)$, the parabolic weights of $\mathcal C_1^{-1}(E, \theta)$ at $x_1$ are $\left\{\left\langle\frac{pd}{N}\right\rangle, \left\langle\frac{p(N-d)}{N}\right\rangle\right\}$ by \cite[Proposition 2.18]{KrSh20}.
Let
\[d'\coloneqq\min\left\{N\left\langle\frac{pd}{N}\right\rangle, N\left\langle\frac{p(N-d)}{N}\right\rangle\right\}.\]
Let $\pi_i (i=1, 2)$ be the two projections of $\mathbb{P}^1_k\times M_{\Hig}^{N, \frac{d}{N}}$. There exists a parabolic Higgs bundle $(E^u, \theta^u)$ on $\mathbb{P}^1_k\times M_{\Hig}^{N, \frac{d}{N}}$ such that
\[E^u=\pi_1^*\left(\mathcal{O}\left(\frac{d}{N}\right)\oplus\mathcal{O}\left(-\frac{d}{N}\right)\right), \quad (E^u, \theta^u)\mid_{\mathbb{P}^1\times\{x\}}\cong (E_x, \theta_x),\]
where $x\in M_{\Hig}^{N, \frac{d}{N}}(k)$ is the $k$-point corresponding to the isomorphism class of parabolic Higgs bundles $[(E_x, \theta_x)]$. Indeed, let $\underline{M}_{\Hig}^{N, \frac{d}{N}}$ be the moduli functor represented by $M_{\Hig}^{N, \frac{d}{N}}$. The parabolic Higgs bundle $(E^u, \theta^u)$ corresponds to the identity map in $\underline{M}_{\Hig}^{N, \frac{d}{N}}\left(M_{\Hig}^{N, \frac{d}{N}}\right)=\Hom\left(M_{\Hig}^{N, \frac{d}{N}}, M_{\Hig}^{N, \frac{d}{N}}\right)$. The parabolic Higgs bundle $(E^u, \theta^u)$ is called the \emph {universal parabolic Higgs bundle}.

Applying the inverse Cartier functor on the universal Higgs bundle, we obtain a de Rham bundle
\[(V^u, \nabla^u)\coloneqq \mathcal C_1^{-1}(E^u, \theta^u)\]
on $\mathbb{P}^1_k\times M_{\Hig}^{N, \frac{d}{N}}$. For any $x\in M_{\Hig}^{N, \frac{d}{N}}(k)$,
\[(V_x, \nabla_x)\coloneqq (V^u, \nabla^u)\mid_{\mathbb{P}^1_k\times \{x\}}\cong \mathcal C_1^{-1}(E_x, \theta_x),\]
where$[(E_x, \theta_x)]$ is the isomorphism class corresponding to $x$.

\[(V^u, \nabla^u)\mid_{\mathbb{P}^1\times\{y\}}\cong (V_y, \nabla_y),\]

\begin{proposition}
The map
\[\lambda: M_{\Hig}^{N, \frac{d}{N}}\to \frac{1}{N}\mathbb{Z}, \quad x\mapsto \lambda_x\]
 is upper semicontinuous.
\end{proposition}
\begin{proof}
Let
\[U_{\alpha}\coloneqq\{x\in M_{\Hig}^{N, \frac{d}{N}}:\lambda_x\leq \alpha\}\]with $\alpha\in \frac{1}{N}\mathbb{Z}$.
By the definition of $\lambda_x$ and $U_\alpha$, $\lambda_x\in U_\alpha$ if and only if $V_x\cong \mathcal{O}(\lambda_x)\oplus \mathcal{O}(-\lambda_x)$ with $0<\lambda_x\leq \alpha$, which is equivalent to
\[H^0\left(\mathbb{P}^1, V_x \otimes \mathcal{O}\left(-\alpha-\frac{1}{N}\right)\right)=0.\]
Considering the bundle $V^u\otimes \pi_1^*\mathcal{O}\left(-\alpha-\frac{1}{N}\right)$ on $\mathbb{P}^1\times M_{\Hig}^{N, \frac{d}{N}}$, its fiber of $\pi_2$ at $x$ is isomorphic to $V_x\left(-\alpha-\frac{1}{N}\right)$. The map
\[M_{\Hig}^{N, \frac{d}{N}} \to \mathbb{Z}_{\geq 0}, \quad x \mapsto \dim H^0\left(\mathbb{P}^1, V^u\otimes \pi_1^*\mathcal{O}\left(-\alpha-\frac{1}{N}\right)\bigg|_{\mathbb{P}^1\times \{x\}}\right)\]
is upper continuous by the semicontinuity theorem. Therefore, $U_\alpha$ is the inverse image of $0$ under this map, which is open.
\end{proof}


\subsubsection{Simpson's Hodge filtration and selfmap}
In the following, we construct a subsheaf $\mathcal{F}_{\frac{d'}{N}}$ of $V^u\mid_{\mathbb{P}^1_{k}\times U_{\frac{d'}{N}}}$ on $\mathbb{P}^1_{k}\times U_{\frac{d'}{N}}$ such that there is a open subset $U\subseteq U_{\frac{d'}{N}}$, for any $x\in U$, $\mathcal{F}_{\frac{d'}{N}}\cong \mathcal{O}\left(\frac{d'}{N}\right)$. Let
\[(V_{\frac{d'}{N}}, \nabla_{\frac{d'}{N}})\coloneqq (V^u, \nabla^u)\mid_{\mathbb{P}^1_{k}\times U_{\frac{d'}{N}}}, \quad L\coloneqq \pi_{2*}\left(\left(V_{\frac{d'}{N}}\otimes \pi_1^*\mathcal{O}\left(-\frac{d'}{N}\right)\right)_0\right).\]
Then for any $x\in U_{\frac{d'}{N}}$,
\[L_x/\mathfrak{m}_x L_x=H^0\left(\mathbb{P}^1_k, \left(V_{\frac{d'}{N}}\otimes p_1^*\mathcal{O}\left(-\frac{d'}{N}\right)\right)_0\bigg|_{\mathbb{P}^1_k\times \{x\}}\right)=H^0\left(\mathbb{P}^1_k, \mathcal{O}\oplus \mathcal{O}(-1)\right)=k,\]
where $L_x$ is the stalk of $L$ at $x$ and $\mathfrak{m}_x$ is the maximal ideal of the local ring at $x$. By Grauert’s theorem( \cite[III.12.9]{Har77}), $L$ is a line bundle on $U_{\frac{d'}{N}}$.
Let
\[\mathcal{F}=p_1^{*}\mathcal{O}\left(\frac{d'}{N}\right)\otimes p_2^{*}L,\]
$\mathcal{F}$ is a parabolic line bundle on $\mathbb{P}^1\times U_{\frac{d'}{N}}$ and $\mathcal{F}|_{\mathbb{P}^1\times\{x\}}\cong \mathcal{O}\left(\frac{d'}{N}\right)$ for any $x\in U_{\frac{d'}{N}}$. There is a natural morphism from $\mathcal{F}$ to $V_{\frac{d'}{N}}$:
\begin{equation*}
\begin{split}
\mathcal{F}=&\pi_1^{*}\mathcal{O}\left(\frac{d'}{N}\right)\otimes \pi_2^{*} \pi_{2*}\left(\left(V_{\frac{d'}{N}}\otimes \pi_1^*\mathcal{O}\left(-\frac{d'}{N}\right)\right)_0\right)\\
& \to \pi_1^{*}\mathcal{O}\left(\frac{d'}{N}\right)\otimes \left(V_{\frac{d'}{N}}\otimes \pi_1^*\mathcal{O}\left(-\frac{d'}{N}\right)\right)_0\to V_{\frac{d'}{N}}.
\end{split}
\end{equation*}
Its image is a parabolic line bundle on a Zariski open subset $\mathbb{P}^1\times U$ of $\mathbb{P}^1\times U_{\frac{d'}{N}}$.

Let $\underline{U}$ be the sub moduli functor of $M_{\Hig}^{N, \frac{d'}{N}}$, which is represented by $U$. Denote $\mathcal{F}_{\frac{d'}{N}}$ the restriction of the image of $\mathcal{F}\to V_{\frac{d'}{N}}$ on $U$. By taking grading sheaf with respect to $\mathcal{F}_{\frac{d'}{N}}$ after applying the inverse Cartier functor, we can define a natural transform
\[\Gr\circ \mathcal C_1^{-1} :\underline{U}\to \underline{M}_{\Hig}^{N, \frac{d'}{N}}.\]
Indeed, for any $k$-scheme $S$ and $f\in \underline{U}(S)= \Hom(S, U)$, $f$ is associated with the parabolic Higgs bundle $f^*\left((E^u, \theta^u)\mid_{\mathbb{P}^1_k\times U}\right)$ on $S$. By taking grading sheaf with respect to $f^*\mathcal{F}_{\frac{d'}{N}}$ after applying the inverse Cartier functor, we obtain a parabolic Higgs bundle over $(\mathbb{P}_S, \sum D_i )/S$, where $D_i=\{x_i\}\times S$. Thus, we can define a compatible family of maps
\[\Gr\circ \mathcal C_1^{-1}:\underline{U}(S)\to \underline{M}_{\Hig}^{N, \frac{d'}{N}}(S), \quad \left[f^*\left((E^u, \theta^u)\mid_{\mathbb{P}^1_k\times U}\right)\right]\mapsto \left[\Gr_{f^*\mathcal{F}_{\frac{d'}{N}}}f^*\left((V_{\frac{d'}{N}}, \nabla_{\frac{d'}{N}})\mid_{\mathbb{P}^1_k\times U}\right)\right],\]
which is a natural transform between the two functors. This also define a morphism between schemes:
\[\varphi\coloneqq\Gr\circ \mathcal C_1^{-1}: U \to M_{\Hig}^{N, \frac{d'}{N}}\]
 by the Yoneda lemma.

\begin{theorem}\label{Thm_SelfmapDominant}
Let $\frac{d}{N}\in J_N$ and $d'=\min\left\{N\left\langle\frac{pd}{N}\right\rangle, N\left\langle\frac{p(N-d)}{N}\right\rangle\right\}$, then the rational map $\varphi:M_{\Hig}^{N, \frac{d}{N}}\dashrightarrow M_{\Hig}^{N, \frac{d'}{N}}$ is dominant.
\end{theorem}
\begin{proof}
To prove that $\varphi$ is dominant, we use induction on $m$. When $m=3$, $M_{\Hig}^{N}$ is a single point, and the proposition holds trivially. Assume that the proposition holds for $m-1$. Set $Z\coloneqq \overline{\text{Im}(\varphi)}$. $Z$ is irreducible, since $Z=\overline{\varphi(U)}$ and $U\subseteq M^{N, \frac{d}{N}}$ is irreducible. If $\varphi$ is not dominant, $Z$ is a proper subscheme of $M_{\Hig}^{N, \frac{d'}{N}}\cong \mathbb{P}^{m-3}$, then $\dim Z\leq m-4$.
Let $M(\hat{x}_m)$ be the moduli space of semistable graded parabolic Higgs bundles of rank $2$ and degree $0$ over $\mathbb{P}^1$ such that the Higgs fields have $m-1$ poles ${x_1, x_2, \ldots, x_{m-1}}$ and the parabolic weights are in $J_N$ at $x_1$ and 0 at other points. There is a natural embedding $M(\hat{x}_m)\to M_{Hig}^{N, \frac{d}{N}}$, which is defined by forgetting $x_m$. The embedding is not surjective. By the assumption that $\varphi$ is dominant for $m-1$, thus, $\mathbb{P}^{m-4}\cong \overline{\varphi(M(\hat{x}_m))}$ is a proper subscheme of $Z$, which contradicts the condition that $Z$ is irreducible and $\dim Z\leq m-4$. Thus, $\varphi$ is dominant for $m$, and the proof is complete.
\end{proof}

Since $p\nmid N$ and $\gcd(d, N)=1$, there always exists a positive integer $k$ such that
\[\left\{\left\langle\frac{p^kd}{N}\right\rangle, \left\langle\frac{p^k(N-d)}{N}\right\rangle\right\}=\left\{\left\langle\frac{d}{N}\right\rangle, \left\langle\frac{N-d}{N}\right\rangle\right\}.\]
 Then the $k$-times composition $\varphi^k$ of $\varphi$ is a dominant rational map from $M_{\Hig}^{N, \frac{d}{N}}$ to itself. Using the same method as in \cite[Theorem 4.4]{SYZ22} and applying a theorem of Hrushovski (\cite[Corollary 1.2]{Hru04}), we can prove the following corollary of \autoref{Thm_SelfmapDominant}:
\begin{corollary}\label{thm_PeriodicPointsDense}
For any $\frac{d}{N}\in J_N$, let $k$ be the smallest positive integer such that
\[\left\{\left\langle\frac{p^kd}{N}\right\rangle, \left\langle\frac{p^k(N-d)}{N}\right\rangle\right\}=\left\{\left\langle\frac{d}{N}\right\rangle, \left\langle\frac{N-d}{N}\right\rangle\right\}.\]
 the set of periodic points of $\varphi^k$ is Zariski dense in $M_{\Hig}^{N, \frac{d}{N}}$. Combining this proposition with \autoref{thm_equFunctorHdRF&FFMod}, we obtain infinitely many parabolic Fontaine-Faltings modules.
\end{corollary}
As an example, assume that $m=4$ and $N=5$. In this case,
\[M_{\Hig}^5=\coprod_{i=1}^4M_{\Hig}^{5, \frac{i}{5}}\]
 with
$M_{\Hig}^{5, \frac{1}{5}}\cong \mathbb{P}^1_{k}$, $M_{\Hig}^{5, \frac{2}{5}}\cong \mathbb{P}^1_{k}$, $M_{\Hig}^{5, \frac{3}{5}}\cong \Spec k$ and $M_{\Hig}^{5, \frac{4}{5}}\cong \Spec k$.

In the case of $p\equiv 1, 4\mod 5$, we have the diagram of the self map:

\begin{tikzpicture}
\Vertex[x=2, y=4, size=2, label=${\lambda=\frac{1}{5}}$, fontscale=1.5]{A}
\Vertex[x=6, y=4, size=0, label=${\lambda=\frac{4}{5}}$, position=above, fontscale=1.5]{B}
\Vertex[x=10, y=4, size=2, label=${\lambda=\frac{2}{5}}$, fontscale=1.5]{C}
\Vertex[x=14, y=4, size=0, label=${\lambda=\frac{3}{5}}$, position=above, fontscale=1.5]{D}

\Vertex[x=2, size=2, label=${\lambda=\frac{1}{5}}$, fontscale=1.5]{E}
\Vertex[x=6, size=0, label=${\lambda=\frac{4}{5}}$, position=below, fontscale=1.5]{F}
\Vertex[x=10, size=2, label=${\lambda=\frac{2}{5}}$, fontscale=1.5]{G}
\Vertex[x=14, size=0, label=${\lambda=\frac{3}{5}}$, position=below, fontscale=1.5]{H}

\Edge[Direct](A)(E)
\Edge[Direct](A)(F)
\Edge[Direct, style={dashed}](B)(E)
\Edge[Direct, style={dashed}](B)(F)
\Edge[Direct](C)(G)
\Edge[Direct](C)(H)
\Edge[Direct, style={dashed}](D)(G)
\Edge[Direct, style={dashed}](D)(H)

\draw (0, -1.2) rectangle (15, 1.2);
\draw (0, 2.8) rectangle (15, 5.2);
\node at (0, 2) {self map};
\node at (7.5, -1.7) {$M_{\Hig}^5$};
\node at (7.5, 5.7) {$M_{\Hig}^5$};
\end{tikzpicture}

In the case of $p\equiv 2, 3\mod 5$, we have the diagram of the self map:

\begin{tikzpicture}
\Vertex[x=2, y=4, size=2, label=${\lambda=\frac{1}{5}}$, fontscale=1.5]{A}
\Vertex[x=6, y=4, size=0, label=${\lambda=\frac{4}{5}}$, position=above, fontscale=1.5]{B}
\Vertex[x=10, y=4, size=2, label=${\lambda=\frac{2}{5}}$, fontscale=1.5]{C}
\Vertex[x=14, y=4, size=0, label=${\lambda=\frac{3}{5}}$, position=above, fontscale=1.5]{D}

\Vertex[x=2, size=2, label=${\lambda=\frac{2}{5}}$, fontscale=1.5]{E}
\Vertex[x=6, size=0, label=${\lambda=\frac{3}{5}}$, position=below, fontscale=1.5]{F}
\Vertex[x=10, size=2, label=${\lambda=\frac{1}{5}}$, fontscale=1.5]{G}
\Vertex[x=14, size=0, label=${\lambda=\frac{4}{5}}$, position=below, fontscale=1.5]{H}

\Edge[Direct](A)(E)
\Edge[Direct](A)(F)
\Edge[Direct, style={dashed}](B)(E)
\Edge[Direct, style={dashed}](B)(F)
\Edge[Direct](C)(G)
\Edge[Direct](C)(H)
\Edge[Direct, style={dashed}](D)(G)
\Edge[Direct, style={dashed}](D)(H)

\draw (0, -1.2) rectangle (15, 1.2);
\draw (0, 2.8) rectangle (15, 5.2);
\node at (0, 2) {self map};
\node at (7.5, -1.7) {$M_{\Hig}^5$};
\node at (7.5, 5.7) {$M_{\Hig}^5$};
\end{tikzpicture}

In both cases, the selfmaps between the maximal components are rational and dominant. The components of the image of $M_{\Hig}^{5, \frac{3}{5}}$ and $M_{\Hig}^{5, \frac{4}{5}}$ depend on the choice of $\{x_1, x_2, x_3, x_4\}$.

\subsection{Parabolic Higgs bundles associated with families of elliptic curves}
In \cite{KrSh20}, Krishnamoorthy and Sheng proved that any motivic (coming from geometry origin) Higgs bundle is periodic. In this subsection, we determine the periodic Higgs bundles arising from families of elliptic curves with four singular fibers.


Let $m=4$ and $D=0+1+\lambda+\infty$ with $\lambda\neq 0, 1, \infty$. Let $U=\mathbb{P}^1_{\mathbb{C}}-\{0, 1, \lambda, \infty\}$.
Given a family of elliptic curves $f:X \to \mathbb{P}^1_{\mathbb{C}}$ with 4 singular fibers at $\{0, 1, \lambda, \infty\}$. Let $X^{\circ}=f^{-1}(U)$. Steenbrink \cite{St76} shows that there exists a unique logarithmic connection $\nabla$ on the rank $2$ bundle $V\coloneqq R^1f_{*}\Omega^*_{X/\mathbb{P}^{1}_{\mathbb{C}}}(\log D)$ such that all the eigenvalues of the residues are contained in $[0, 1)\cap \mathbb{Q}$ and the restriction of $(V, \nabla)$ on $U$ is the Gau\ss-Manin connection $R^1f_{*}\left(\Omega^*_{X^{\circ}/U}, d\right)$ associated with $f$. There is a natural filtration
\[0\subseteq f_{*}\Omega^1_{X^/\mathbb{P}^{1}_{\mathbb{C}}}(\log D)\subseteq R^1f_{*}\Omega^*_{X^/\mathbb{P}^{1}_{\mathbb{C}}}(\log D)\]
satisfying the Griffiths transversality. \cite[Lemma 3.3]{IySi07}  implies that there is a  parabolic filtered de Rham bundle corresponding to $(V, \nabla)$. By taking the grading with respect to the filtration, we obtain a rank $2$ motivic graded parabolic Higgs bundle
\[(E, \theta)\coloneqq \left(f_{*}\Omega^1_{X^/\mathbb{P}^{1}_{\mathbb{C}}}(\log D)\oplus R^1f_{*}\mathcal{O}_{X}, \Gr(\nabla)\right)\]
arising from the family $f$.

In \cite{Her91}, Herfurtner classified all non-isotrivial families of elliptic curves over $\mathbb{P}_{\mathbb{C}}^1$ with four singular fibers. For any family of elliptic curves, the type of a singular fiber must be one of
\[\{I_n(n\geq 1), II, III, IV, I_n^*(n\geq 0), II^*, III^*, IV^*\}.\]
 \cite[Table 1]{Her91} lists the local monodromy of all these types of singular fibers. The eigenvalues of the local monodromy of these types of singular fibers are
\[\{1, e^{\pm\frac{\pi i}{3}}, e^{\pm\frac{\pi i}{2}}, e^{\pm\frac{2\pi i}{3}}, -1, e^{\pm\frac{\pi i}{3}}, e^{\pm\frac{\pi i}{2}}, e^{\pm\frac{2\pi i}{3}}\}.\]
Thus the eigenvalues of the residues of the logarithmic de Rham bundle arising from the family at the singular fiber are
\[\left\{0, \left(\frac{1}{6}, \frac{5}{6}\right), \left(\frac{1}{4}, \frac{3}{4}\right), \left(\frac{1}{3}, \frac{2}{3}\right), \left(\frac{1}{2}, \frac{1}{2}\right), \left(\frac{1}{6}, \frac{5}{6}\right), \left(\frac{1}{4}, \frac{3}{4}\right), \left(\frac{1}{3}, \frac{2}{3}\right)\right\}.\]
They are also equal to the parabolic weights of the parabolic Higgs bundle arsing from a family of elliptic curves at the singular fibers.

Let $(E, \theta)$ be a motivic parabolic Higgs bundle arising from a family of elliptic curves $f:X\to \mathbb{P}^1_{\mathbb{C}}$ with four singular fibers at $\{0, 1, \lambda, \infty\}$, with the parabolic weights in $\frac{1}{N}\mathbb{Z}_{>0}$ at $\infty$ and zero at $\{0, 1, \lambda\}$. By the discussion above, the types of the singular fibers at $\{0, 1, \lambda\}$ are of the form $I_n(n\geq 1)$, and the possible choices of parabolic weights at $\infty$ are $\left(\frac{1}{2}, \frac{1}{2}\right)$, $\left(\frac{1}{3}, \frac{2}{3}\right)$, $\left(\frac{1}{4}, \frac{3}{4}\right)$ and $\left(\frac{1}{6}, \frac{5}{6}\right)$. When the parabolic weights are $\left(\frac{1}{2}, \frac{1}{2}\right)$, the type of the singular fiber at $\infty$ is $I^{*}_{n}(n\geq 0)$; when the parabolic weights are $\left(\frac{1}{3}, \frac{2}{3}\right)$, the type of the singular fiber at $\infty$ is $IV$ or $IV^*$; when the parabolic weights are $\left(\frac{1}{4}, \frac{3}{4}\right)$, the type of the singular fiber at $\infty$ is $III$ or $III^*$; when the parabolic weights are $\left(\frac{1}{6}, \frac{5}{6}\right)$, the type of the singular fiber at $\infty$ is $II$ or $II^*$.

Let $X, Y$ be the homogeneous coordinates of $\mathbb{P}^1_{\mathbb{C}}$. In \autoref{Table1}, up to transfer the $``*"$ and $``*"$ the singular fibers in pairs (\cite[p.324]{Her91}), we list all the families of elliptic curves with four singular fibers, whose parabolic weights of the associated parabolic Higgs bundles are in $\frac{1}{N}\mathbb{Z}_{>0}$ at a single point.

\begin{table}[htp]\label{Table1}
\caption{}
\begin{tabular}{|c|c|c|c|}
\hline
$N$& fiber combination & singular locus& $j$-invariant at $[X:Y]$ \\ \hline
\multirow{4}{*}{2} & $I_1I_1I_1I_3^*$ & $(\omega_1, \omega_2, \infty, 0)$ & $-\frac{4}{(\alpha-1)^2} \frac{\left(X^2+2 \alpha X Y+Y^2\right)^3}{X^3 Y\left[12 X^2+3\left(3 \alpha^2+6 \alpha-1\right) X Y+4(\alpha+2) Y^2\right]}, \quad \alpha \neq-2, -\frac{5}{3}, 1$ \\ \cline{2-4}
 & $I_1I_1I_2I_2^*$ & $(\omega_3, \omega_4, \infty, 0)$ & $\frac{4}{(2-\alpha)^2} \frac{\left(X^2+\alpha X Y+Y^2\right)^3}{X^2 Y^2\left[3 X^2+2(2 \alpha-1) X Y+3 Y^2\right]}, \quad\alpha\neq -1, 2$ \\ \cline{2-4}
 &$I_0^*I_4I_1I_1$ & $(\lambda, 1, \infty, 0)$&$\frac{1}{108} \frac{\left(X^2+14 X Y+Y^2\right)^3}{X Y(X-Y)^4} \quad \lambda \neq 0, 1, \infty$ \\ \cline{2-4}
 &$I_0^*I_2I_2I_2$ & $(\lambda, 1, \infty, 0)$& $\frac{4}{27} \frac{\left(X^{2}-X Y+Y^{2}\right)^{3}}{X^{2} Y^{2}(X-Y)}, \quad \lambda\neq 0, 1, \infty$ \\ \hline
\multirow{5}{*}{3} & $I_1I_1I_2IV^*$ & $(\omega_5, \omega_6, \infty, 0)$ & $\frac{X(X+2 \alpha Y)^3}{Y^2\left[\left(3 \alpha^2-2\right) X^2+2 \alpha\left(4 \alpha^2-3\right) X Y-Y^2\right]}, \quad \alpha \neq 0, \pm \sqrt{\frac{1}{2}}, \pm \sqrt{\frac{2}{3}}$ \\ \cline{2-4}
 & $I_1I_1I_6IV$& $(1, -1, \infty, 0)$ & $\frac{1}{64}\frac{X^2(9X^2-8Y^2)^3}{Y^6(X^2-Y^2)}$ \\ \cline{2-4}
 & $I_1I_2I_5IV$ & $\left(-\frac{27}{4}, -\frac{1}{2}, \infty, 0\right)$ & $-\frac{4}{27} \frac{X^{2}\left(X^{2}+8 X Y+10 Y^{2}\right)^{3}}{Y^{5}(2 X+Y)^{2}(4 X+27 Y)}$ \\ \cline{2-4}
 & $I_3I_3I_2IV$ & $(\infty, 0, -1, 1)$ & $-\frac{1}{2^{12}} \frac{(X-Y)^{2}\left(9 X^{2}+14 X Y+9 Y^{2}\right)^{3}}{X^{3} Y^{3}(X+Y)^{2}}$ \\ \cline{2-4}
 & $I_0^*I_3I_1II$& $(\lambda, 0, \infty, 1)$& $-\frac{1}{64} \frac{(X-Y)(X-9 Y)^{3}}{X^{3} Y}, \quad \lambda\neq 0, 1, \infty$ \\ \hline
\multirow{6}{*}{4} & $I_1I_1I_1III^*$ & $(\omega_7, \omega_8, \infty, 0)$ & $\frac{(X+\alpha Y)^{3}}{Y\left[\left(3 \alpha-2\right) X^{2}-\left(3 \alpha^{2}-1\right) X Y+\alpha^3 Y\right]}, \quad \alpha \neq -\frac{1}{3}, 0, \frac{2}{3}, 1 $ \\ \cline{2-4}
 & $I_1I_1I_7III$ & $(c_1, c_2, \infty, 0)$ & $\frac{4}{27} \frac{\left(X^{3}+4 X^{2} Y+10 X Y^{2}+6Y^{3}\right)^{3}}{Y^{7}(4X^2+13XY+32Y^2)}$ \\ \cline{2-4}
 & $I_1I_2I_6III$ & $(4, 1, \infty, 0)$ & $\frac{4}{27} \frac{\left(X^{3}-6 X^{2} Y+9 X Y^{2}-3 Y^{3}\right)^{3}}{Y^{6}(X-Y)^{2}(X-4 Y)}$ \\ \cline{2-4}
 & $I_1I_3I_5III$ & $(-\frac{25}{3}, 0, \infty, \frac{1}{5})$ & $-\frac{25}{2^{14} \cdot 3^{3}} \frac{\left(5 X^{3}+45 X^{2} Y+39 X Y^{2}-25 Y^{3}\right)^{3}}{X^{3} Y^{5}(3 X+25 Y)} $ \\ \cline{2-4}
 & $I_2I_3I_4III$ & $(-\frac{1}{3}, 0, \infty, 1)$ & $\frac{1}{108} \frac{\left(16 X^{3}-3 X Y^{2}-Y^{3}\right)^{3}}{X^{3} Y^{4}(3 X+Y)^{2}}$ \\ \cline{2-4}
 &$ I_0^*I_2I_1III$ & $(\lambda, 0, \infty, 1)$&$-\frac{1}{27} \frac{(X-4 Y)^{3}}{X^{2} Y}, \quad \lambda\neq 0, 1, \infty$ \\ \hline
\multirow{5}{*}{6} & $I_1I_1I_8II$ & $(c_3, c_4, \infty, 0)$ & $-\frac{4}{27} \frac{X\left(X^{3}-6 X^{2} Y+15 X Y^{2}-12 Y^{3}\right)^{3}}{Y^{8}\left(3 X^{2}-14 X Y+27 Y^{2}\right)}$ \\ \cline{2-4}
 & $I_1I_2I_7II$ & $(-\frac{9}{4}, -\frac{8}{9}, \infty, 0)$ & $-4 \frac{X\left(9 X^{3}+36 X^{2} Y+42 X Y^{2}+14 Y^{3}\right)^{3}}{Y^{7}(9 X+8 Y)^{2}(4 X+9 Y)}$ \\ \cline{2-4}
 & $I_1I_4I_5II$ & $(-10, 0, \infty, \frac{1}{8})$ & $-\frac{1}{2^{3} \cdot 3^{12}}\frac{(8X-Y)(8X^3+87X^2Y+96XY^2-63Y^3)^3}{X^4Y^5(X+10Y)}$ \\ \cline{2-4}
 & $I_2I_3I_5II$ & $(-\frac{5}{9}, 0, \infty, 3)$ & $-\frac{1}{2^{14} \cdot 3} \frac{(X-3 Y)\left(81 X^{3}-9 X^{2} Y-53 X Y^{2}-27 Y^{3}\right)^{3}}{X^{3} Y^{5}(9 X+5 Y)^{2}} $ \\ \cline{2-4}
 &$ I_0^*I_1I_1 IV$ & $(\lambda, 0, \infty, 1)$&$-\frac{1}{4} \frac{(X-Y)^{2}}{X Y}, \quad \lambda \neq 0, 1, \infty$\\ \hline
\end{tabular}
\end{table}

In this table,
\begin{equation*}
\begin{split}
c_{1, 2}&=\frac{1}{4}\left(\frac{1 \pm i \sqrt{7}}{2}\right)^{7}, \quad c_{3, 4}=-\frac{1}{3}(1 \pm i \sqrt{2})^{4}, \\
\omega_{1, 2}&=-\frac{1}{8}\left(3 \alpha^{2}+6 \alpha-1 \pm \sqrt{\frac{1}{3}(\alpha-1)(3 \alpha+5)^{3}}\right), \\
\omega_{3, 4}&=-\frac{1}{3}\left(2\alpha-1\pm2\sqrt{\alpha^2-\alpha-2}\right), \\
\omega_{5, 6}&=-\frac{1}{3 \alpha^{2}-2}\left(4\alpha^3-3\alpha \pm \sqrt{2\left(2 \alpha^{2}-1\right)^{3}}\right), \\
\omega_{7, 8}&=-\frac{1}{6 \alpha-4}\left(3 \alpha^{2}-1 \pm \sqrt{(3 \alpha+1)(1-\alpha)^{3}}\right).
\end{split}
\end{equation*}
For any family in the table, by applying an automorphism of $\mathbb{P}^1_{\mathbb{C}}$,
we can transfer the singular locus to $\{0, 1, \lambda, \infty\}$.

The tuple $(E, \theta, \mathbb{P}^1_{\mathbb{C}})$ admits a spreading out $(\mathscr{E}, \Theta, \mathbb{P}^1_{A})$ over $A$, where $A$ is a finitely generated sub $\mathbb{Z}$-algebra of $\mathbb{C}$. \cite[Proposition 3.3]{LSZ19} implies that there is a nonempty open subset $U\subseteq \Spec A$ such that for any maximal ideal $\mathfrak{p}\in U$, the parabolic Higgs bundle
$$(E, \theta)_{\mathfrak{p}}\coloneqq (\mathscr{E}, \Theta)\otimes \overline{k_{\mathfrak{p}}}$$ is $1$-periodic, where $k_{\mathfrak{p}}$ is the residue field of $A$ at $\mathfrak{p}$. Fix a maximal ideal $\mathfrak{p}\in U$, assume that $\Char k_{\mathfrak{p}}=p>3$. The isomorphism class $[(E, \theta)_{\mathfrak{p}}]\in M_{\Hig}^N(\overline{k_\mathfrak{p}})$ is a $1$-periodic point of the selfmap.


\begin{thebibliography}{LYZ25}

\bibitem[Bis97]{Bis97}
Indranil Biswas.
\newblock Parabolic bundles as orbifold bundles.
\newblock {\em Duke Mathematical Journal}, 88:305--325, 1997.

\bibitem[BM24]{BiFr24}
Indranil Biswas and Francois-Xavier Machu.
\newblock Equivariant vector bundles over the complex projective line, 2024.

\bibitem[DI87]{DeIl87}
Pierre Deligne and Luc Illusie.
\newblock Rel{\`e}vements modulo p 2 et d{\'e}composition du complexe de de
  rham.
\newblock {\em Inventiones mathematicae}, 89:247--270, 1987.

\bibitem[EG20]{EsGr20}
H\'{e}l\`ene Esnault and Michael Groechenig.
\newblock Rigid connections and {$F$}-isocrystals.
\newblock {\em Acta Math.}, 225(1):103--158, 2020.

\bibitem[EV92]{EsVi92}
H\'{e}l\`ene Esnault and Eckart Viehweg.
\newblock {\em Lectures on vanishing theorems}, volume~20 of {\em DMV Seminar}.
\newblock Birkh\"{a}user Verlag, Basel, 1992.

\bibitem[Fal89]{Fal89}
Gerd Faltings.
\newblock Crystalline cohomology and {$p$}-adic {G}alois-representations.
\newblock In {\em Algebraic analysis, geometry, and number theory ({B}altimore,
  {MD}, 1988)}, pages 25--80. Johns Hopkins Univ. Press, Baltimore, MD, 1989.

\bibitem[Fal99]{Fal99}
Gerd Faltings.
\newblock Integral crystalline cohomology over very ramified valuation rings.
\newblock {\em J. Amer. Math. Soc.}, 12(1):117--144, 1999.

\bibitem[FL82]{FoLa82}
Jean-Marc Fontaine and Guy Laffaille.
\newblock Construction de repr\'esentations {$p$}-adiques.
\newblock {\em Ann. Sci. \'Ecole Norm. Sup. (4)}, 15(4):547--608 (1983), 1982.

\bibitem[Har77]{Har77}
Robin Hartshorne.
\newblock {\em Algebraic Geometry}.
\newblock Graduate Texts in Mathematics. Springer, 1977.

\bibitem[Her91]{Her91}
Stephan Herfurtner.
\newblock Elliptic surfaces with four singular fibres.
\newblock {\em Mathematische Annalen}, 291(2):319--342, 1991.

\bibitem[HL10]{HuLe10}
Daniel Huybrechts and Manfred Lehn.
\newblock {\em The geometry of moduli spaces of sheaves}.
\newblock Cambridge Mathematical Library. Cambridge University Press,
  Cambridge, second edition, 2010.

\bibitem[Hru04]{Hru04}
Ehud Hrushovski.
\newblock The elementary theory of the {F}robenius automorphisms, 2004.
\newblock arXiv:math/0406514.

\bibitem[IS07]{IySi07}
Jaya N.~N. Iyer and Carlos~T. Simpson.
\newblock A relation between the parabolic {C}hern characters of the de {R}ham
  bundles.
\newblock {\em Math. Ann.}, 338(2):347--383, 2007.

\bibitem[KS20]{KrSh20}
Raju Krishnamoorthy and Mao Sheng.
\newblock Periodic de {R}ham bundles over curves, 2020.
\newblock arXiv:2011.03268.

\bibitem[Lan15]{Lan15}
Adrian Langer.
\newblock Bogomolov's inequality for {H}iggs sheaves in positive
  characteristic.
\newblock {\em Invent. Math.}, 199(3):889--920, 2015.

\bibitem[LL23]{LaLi23}
Yeuk Hay~Joshua Lam and Daniel Litt.
\newblock Geometric local systems on the projective line minus four points,
  2023.
\newblock arXiv:2305.11314.

\bibitem[LSZ19]{LSZ19}
Guitang Lan, Mao Sheng, and Kang Zuo.
\newblock Semistable {H}iggs bundles, periodic {H}iggs bundles and
  representations of algebraic fundamental groups.
\newblock {\em J. Eur. Math. Soc. (JEMS)}, 21(10):3053--3112, 2019.

\bibitem[LYZ25]{LYZ25}
Zhenmou Liu, Jinbang Yang, and Kang Zuo.
\newblock Logarithmic crystalline representations, 2025.
\newblock arXiv:2504.14246.

\bibitem[MY92]{Mar92}
M.~Maruyama and K.~Yokogawa.
\newblock Moduli of parabolic stable sheaves.
\newblock {\em Mathematische Annalen}, 293(1):77--99, Dec 1992.

\bibitem[OV07]{OgVo07}
A.~Ogus and V.~Vologodsky.
\newblock Nonabelian {H}odge theory in characteristic {$p$}.
\newblock {\em Publ. Math. Inst. Hautes \'Etudes Sci.}, 106:1--138, 2007.

\bibitem[Sim97]{Sim97}
Carlos Simpson.
\newblock The {H}odge filtration on nonabelian cohomology.
\newblock In {\em Algebraic geometry---{S}anta {C}ruz 1995}, volume~62 of {\em
  Proc. Sympos. Pure Math.}, pages 217--281. Amer. Math. Soc., Providence, RI,
  1997.

\bibitem[Ste76]{St76}
Joseph Steenbrink.
\newblock Limits of hodge structures.
\newblock {\em Inventiones mathematicae}, 31(3):229--257, 1976.

\bibitem[SYZ22]{SYZ22}
Ruiran Sun, Jinbang Yang, and Kang Zuo.
\newblock Projective crystalline representations of \'{e}tale fundamental
  groups and twisted periodic {H}iggs--de {R}ham flow.
\newblock {\em J. Eur. Math. Soc. (JEMS)}, 24(6):1991--2076, 2022.

\bibitem[YZ23]{YaZu23a}
Jinbang Yang and Kang Zuo.
\newblock Constructing algebraic solutions of painleve vi equation from
  $p$-adic hodge theory and langlands correspondence, 2023.
\newblock arXiv:2301.10054.

\end{thebibliography}
\end{document}